\documentclass[12pt,twoside]{amsart}
\usepackage{mathrsfs}
\usepackage{enumitem}
\usepackage{amsmath, amsfonts, amssymb, amsthm, amscd}
\usepackage[all]{xy}
\usepackage{fancyhdr}
\usepackage{hyperref}
\usepackage{geometry}
\usepackage{makecell}  
\usepackage{lipsum}

\usepackage[mathscr]{eucal}

\theoremstyle{plain}
\newtheorem{thm}{Theorem}[section]

\newtheorem{theorem}[thm]{Theorem}
\newtheorem{lemma}[thm]{Lemma}
\newtheorem{corollary}[thm]{Corollary}
\newtheorem{proposition}[thm]{Proposition}

\newtheorem{definition}[thm]{Definition}

\theoremstyle{remark}

\newtheorem{remark}[thm]{Remark}

\newtheorem{defn-thm}[thm]{Definition-Theorem}
\newtheorem{defn-lem}[thm]{Definition-Lemma}

\renewcommand{\bar}{\overline}

\renewcommand{\phi}{\varphi}

\newcommand{\C}{{\mathbb C}}

\newcommand{\R}{{\mathbb R}}

\newcommand{\Q}{{\mathbb Q}}

\newcommand{\M}{{\mathcal M}}
\newcommand{\T}{{\mathcal T}}

\renewcommand{\tilde}{\widetilde}
\newcommand{\p}{{\Phi}}

\newcommand{\EQ}{{\ \Longleftrightarrow \ }}

\begin{document}

\def\dW{\mbox{diff\:}\times \mbox{Weyl\:}}
\def\End{\operatorname{End}}
\def\Hom{\operatorname{Hom}}
\def\Aut{\operatorname{Aut}}
\def\Diff{\operatorname{Diff}}
\def\im{\operatorname{im}}
\def\tr{\operatorname{tr}}
\def\Pr{\operatorname{Pr}}
\def\Z{\bf Z}
\def\O{\mathcal{O}}
\def\CP{\mathbb{C}\mathbb{P}}
\def\P{\Phi}
\def\TT{\mathcal {T}_m^H}

\def\Q{\bf Q}
\def\R{\bf R}
\def\C{\mathbb{C}}
\def\H{H_{\mathrm{pr}}}
\def\Hil{\mathcal{H}}
\def\proj{\operatorname{proj}}
\def\id{\mbox{id\:}}
\def\a{\mathfrak a}
\def\d{\partial}
\def\tO{\tilde{\Omega}}

\def\b{\beta}
\def\c{\gamma}
\def\p{\partial}
\def\f{\frac}
\def\i{\sqrt{-1}}
\def\t{\tau}
\def\T{\mathcal{T}}
\def\Tan{\mathrm{T}^{1,0}}
\def\aTan{\mathrm{T}^{0,1}}
\def\Kahler{K\"{a}hler\:}
\def\w{\omega}
\def\X{\mathfrak{X}}
\def\K{\mathcal {K}}
\def\m{\mu}
\def\M{\mathcal {M}}
\def\Z{\mathcal {Z}_m}
\def\ZZ{\mathcal {Z}_m^H}

\newcommand{\bp}{\bar{\partial}}

\def\v{\nu}
\def\D{\mathcal{D}}
\def\U{\mathcal {U}}
\def\V{\mathcal {V}}
\def\Omegak{\frac{1}{k!}\bigwedge\limits^k\mu\lrcorner\Omega}
\def\Omegakp{\frac{1}{(k+1)!}\bigwedge\limits^{k+1}\mu\lrcorner\Omega}
\def\Omegakpp{\frac{1}{(k+2)!}\bigwedge\limits^{k+2}\mu\lrcorner\Omega}
\def\Omegakm{\frac{1}{(k-1)!}\bigwedge\limits^{k-1}\mu\lrcorner\Omega}
\def\Omegakmm{\frac{1}{(k-2)!}\bigwedge\limits^{k-2}\mu\lrcorner\Omega}
\def\Omegakk{\Omega_{i_1,i_2,\cdots,i_k}}
\def\Omegakkp{\Omega_{i_1,i_2,\cdots,i_{k+1}}}
\def\Omegakkpp{\Omega_{i_1,i_2,\cdots,i_{k+2}}}
\def\Omegakkm{\Omega_{i_1,i_2,\cdots,i_{k-1}}}
\def\Omegakkmm{\Omega_{i_1,i_2,\cdots,i_{k-2}}}
\def\mukm{\frac{1}{(k-1)!}\bigwedge\limits^{k-1}\mu}
\def\sumk{\sum\limits_{i_1<i_2<,\cdots,<i_k}}
\def\sumkm{\sum\limits_{i_1<i_2<,\cdots,<i_{k-1}}}
\def\sumkmm{\sum\limits_{i_1<i_2<,\cdots,<i_{k-2}}}
\def\sumkp{\sum\limits_{i_1<i_2<,\cdots,<i_{k+1}}}
\def\sumkpp{\sum\limits_{i_1<i_2<,\cdots,<i_{k+2}}}
\def\Omegakb{\Omega_{i_1,\cdots,\bar{i}_t,\cdots,i_k}}
\def\Omegakmb{\Omega_{i_1,\cdots,\bar{i}_t,\cdots,i_{k-1}}}
\def\Omegakpb{\Omega_{i_1,\cdots,\bar{i}_t,\cdots,i_{k+1}}}
\def\Omegakt{\Omega_{i_1,\cdots,\tilde{i}_t,\cdots,i_k}}

\title{Sections of Hodge bundles I: Global theory and applications to period maps}
\author{Kefeng Liu}
\address{Mathematical Sciences Research Center, Chongqing University of Technology, Chongqing 400054, China; \newline
Department of Mathematics,University of California at Los Angeles, Los Angeles, CA 90095-1555, USA}
\email{liu@math.ucla.edu}

\author{Yang Shen}
\address{Mathematical Sciences Research Center, Chongqing University of Technology, Chongqing 400054, China}
\email{syliuguang2007@163.com}
\date{}

\vspace{-20pt}

\begin{abstract}
We study global sections of Hodge bundles arising from two complementary constructions: a deformation-theoretic construction, which yields global geometric consequences for period maps, and a construction from the matrix representation of the image of the period map, which provides an explicit Euclidean realization. 
Combining these perspectives, we prove that the image of the lifted period map on the universal cover is contained in a complex Euclidean subspace of the period domain, thereby giving a partial solution to a conjecture of Griffiths on the global behavior of period maps. 
As an application, we construct a global complex affine structure on the Teichm\"uller space of Calabi--Yau type manifolds.
\end{abstract}

\maketitle
\parskip=5pt
\baselineskip=15pt


\tableofcontents



\setcounter{section}{-1}
\section{Introduction}

In 1968, Griffiths introduced the concept of the period map in Hodge theory \cite{Griffiths1, Griffiths2}. 
His fundamental insight was to interpret points of the moduli space of polarized manifolds in terms of the Hodge structures naturally associated to the underlying manifolds. 
This led to the notion of the period map, which has since become a central tool connecting deformation theory, complex geometry, and Hodge theory. 
While its infinitesimal properties are well understood through
the foundational work of Kodaira and Spencer in \cite{KS12,KS3},
the global behavior of period maps remains far more subtle,
especially in connection with the geometry of moduli spaces,
the Torelli problems, and conjectures of Griffiths in \cite{Griffiths4}.

The present paper develops a new approach to the global study of period maps based on global sections of Hodge bundles constructed simultaneously from the matrix representation of the period map and from deformation theory. 
The key observation is that these two apparently different constructions coincide: Lie-theoretic sections arising from the image of the period map extend holomorphically to globally defined sections obtained via Beltrami differentials and harmonic theory on the base manifold. 
This provides a conceptual bridge between the algebraic geometry of period domains and analytic deformation theory, and leads to several applications, including a partial solution to Griffiths’ conjecture, Conjecture~10.1 in~\cite{Griffiths4}, on the global behavior of period maps and the construction of a global complex affine structure on Teichm\"uller spaces of Calabi--Yau type manifolds.

We now present the main results of this paper.

Recall that the period domain $D$, defined as the set of polarized Hodge structures $$\left(H=\bigoplus_{p+q=n}H^{p,q}\right)$$ of certain type on the complex vector space $H$, can be realized as a flag domain $D = G_\mathbb{R} / V$ in the flag manifold $\check{D}= G_{\C} / B$. Here $G_\C$ is a complex Lie group, $B$ is a parabolic subgroup of $G_\C$, $G_\mathbb{R}$ is a real Lie subgroup of $G_\C$ and $V = B \cap G_{\mathbb{R}}$ is a compact Lie subgroup of $G_{\mathbb{R}}$.

In Hodge theory, the complex Lie subgroup $G_\C$ is the classical group $Sp(N,\C)$ when the weight $n$ is odd and $SO(r,s,\C)$ when the weight $n$ is even. Please see Chapter 4.4 of \cite{CMP} for the details.
The Hodge structure $(H=\oplus_{p+q=n}H_o^{p,q})$ at a fixed point $o$ in $D \subseteq \check{D}$ induces a Hodge structure of weight zero on the Lie algebra $\mathfrak{g}$ of $G_{\mathbb{C}}$ as
$$\mathfrak{g} = \bigoplus_{k \in \mathbb{Z}} \mathfrak{g}^{k, -k}, \quad \mathfrak{g}^{k, -k} = \{X \in \mathfrak{g} \mid X H_o^{p, q} \subseteq H_o^{p+k, q-k}, \ \forall\, p+q=n\}.$$ 
See Section \ref{Lie} for details.
Then the Lie algebra $\mathfrak{b}$ of $B$ can be identified with $\bigoplus_{k \geq 0} \mathfrak{g}^{k, -k}$ and the holomorphic tangent space $\mathrm{T}^{1,0}_o{D}$ of ${D}$ at the base point $o$ is naturally isomorphic to
$$\mathfrak{g} / \mathfrak{b} \simeq \oplus_{k \geq 1} \mathfrak{g}^{-k,k} \triangleq \mathfrak{n}_-.$$
We denote the corresponding unipotent group by 
$$N_- = \exp(\mathfrak{n}_-)$$ 
which is identified with $N_-(o)$, the unipotent orbit of the base point $o$, and is considered a complex Euclidean space inside $\check{D}$.

Let $\Phi:\, S\to \Gamma \backslash D$ be a period map arising from geometry, which means that we have an analytic family $f:\, \X\to
S$ of polarized manifolds over a complex manifold $S$, such that for any $t\in S$, the point $\Phi(t)$, modulo the action
of the monodromy group $\Gamma$, represents the polarized Hodge structure of the $n$-th primitive
cohomology $H_{\mathrm{pr}}^{n}(X_{t},\C)$ of the fiber $X_t=f^{-1}(t)$.

Since the period map is locally liftable, we can lift it to $\P :\,  \tilde S  \to D$ by taking the universal cover $\pi:\, \tilde S \to S$ such that the diagram
\begin{equation*}
\xymatrix{
 \tilde S  \ar[r]^-{\P} \ar[d]^-{\pi} & D \ar[d]^-{\pi} \\
S \ar[r]^-{\Phi} & \Gamma \backslash D
}
\end{equation*}
is commutative.

The lifted period map $\Phi:\, \tilde S \to D$ is equivalent to a decreasing sequence of the Hodge bundles 
$$\mathcal{F}^{n}\subset \cdots \subset \mathcal{F}^{0}=H\times \tilde S$$
equipped with the Gauss--Manin connection $\nabla$ satisfying $\nabla^{2}=0$, and the Griffiths transversality condition
$$
\nabla \mathcal{F}^{p}\subset \mathcal{F}^{p-1}\otimes \Omega^{1}_{\tilde S}.
$$
Therefore, the global study of the period map is reduced to the study of the global behavior of sections of the Hodge bundles.

Let $\Phi:\, \tilde S \to D$ be the lifted period map with base point $t_o\in \tilde S$ such that 
\begin{equation}\label{base Hodge}
o=\Phi(t_o)=\left(H_{\mathrm{pr}}^{n}(X_{t_{o}},\C)=\bigoplus_{p+q=n}H_{\mathrm{pr}}^{p,q}(X_{t_{o}})\right).
\end{equation}
Following an idea proposed by Schmid in a personal communication, we define 
$$\tilde S^{\vee} = \P^{-1}(N_- \cap D).$$ 
It is easy to prove that ${\tilde{S}} \setminus {\tilde{S}}^{\vee}$ is an analytic subvariety of ${\tilde{S}}$ with $\mbox{codim}_\C ({\tilde{S}} \setminus {\tilde{S}}^{\vee}) \ge 1$.

Let 
$$
\eta = \{\eta_{(0)}^T, \eta_{(1)}^T, \cdots, \eta_{(n)}^T\}^T,
$$
be a basis of the Hodge decomposition \eqref{base Hodge} at the base point, where $\eta_{(p)}$, considered as a column vector of elements in $H$, is a basis of $ H_{\mathrm{pr}}^{n-p,p}(X_{t_o}) $ for $ 0 \leq p \leq n $.

The image of the restricted period map $\Phi: \tilde{S}^\vee \to N_- \cap D$,
$$
\P(t) = \left(\begin{array}{ccc}
I & & \left( \Phi^{(p,q)}(t) \right)_{p < q} \\
& \ddots & \\
& & I
\end{array}\right) \in N_-,\,t\in \tilde{S}^\vee
$$
gives a basis
\begin{equation}\label{intr lm section}
  \Omega_{(p)}(t)=\eta_{(p)}+\sum_{k\ge 1} \Phi^{(p,p+k)}(t)\cdot \eta_{(p+k)},
\end{equation}
of the filtration $ F^{n-p} H_{\mathrm{pr}}^{n}(X_{t}, \C) $ modulo $ F^{n-p+1} H_{\mathrm{pr}}^{n}(X_{t}, \C) $ for $ 0 \leq p \leq n $. In other words, $ \Omega_{(0)}(t), \cdots, \Omega_{(p)}(t) $ are holomorphic sections of the Hodge bundle $ \mathcal{F}^{n-p} $ and generate $ \mathcal{F}^{n-p} $ pointwise on $\tilde{S}^\vee$.

Note that we have an induced family $\tilde{\mathcal X}\to \tilde S$ of polarized manifolds over the universal cover $\tilde S$ of $S$ with a polarization $\tilde{\mathcal L}$ on $\tilde{\mathcal X}$. Then the polarization $\tilde{\mathcal L}$ identifies the K\"ahler forms on $X_t$ and $X_{t_o}$ for any $t \in \tilde S$, which implies that the map 
$$\iota_t:\, {\mathrm{T}^{*}}^{1,0}X_t \hookrightarrow {\mathrm{T}_\C^{*}} X_t \stackrel{\cong}{\longrightarrow}{\mathrm{T}_\C^{*}} X_{t_o} \stackrel{\pi^{1,0}}{\longrightarrow}{\mathrm{T}^{*}}^{1,0}X_{t_o}$$
is an isomorphism for any $t\in \tilde S$, where $\pi^{1,0}:\,{\mathrm{T}_\C^{*}} X_{t_o} \to {\mathrm{T}^{*}}^{1,0}X_{t_o}$ and $\pi^{0,1}:\,{\mathrm{T}_\C^{*}} X_{t_o} \to {\mathrm{T}^{*}}^{0,1}X_{t_o}$ are the projection maps. 
Thus, for any $t \in \tilde S$, the Beltrami differential $$\phi(t) =\pi^{0,1}\circ \iota_t^{-1} \in A^{0,1}\left( X_{t_o}, \mathrm{T}^{1,0} X_{t_o} \right)$$ exists with supremum operator norm $\|\phi_t\| < 1 $ (see Definition \ref{son defn}), such that the complex structure of $ X_t $ is determined by $\phi(t)$, cf. Lemma \ref{Lemma-moduli}. 

Then, using the Beltrami differential $\phi(t)$ and its contraction operator
$$i_{\phi(t)}=\phi(t)\lrcorner(\cdot):\, A^{p,q}(X_{t_{o}})\to A^{p-1,q+1}(X_{t_{o}}),$$
as well as the operator $T := \bar\partial^{*}G\partial$ from the harmonic theory on the base fiber $X_{t_{o}}$, we can construct sections $\tO_{(p)}$ of the Hodge bundles $\mathcal{F}^{n-p}$ over $\tilde S$, for $0\le p \le n$, via
\begin{align}
\tO_{(p)}(t)=&\left[\mathbb H_{\mathrm{pr}} \left(e^{i_{\phi(t)}}\left((I+Ti_{\phi(t)})^{-1}\tilde \eta_{(p)}\right)\right)\right]\nonumber \\
=&\eta_{(p)} + \sum_{k\ge 1}\frac{1}{k!}\left[\mathbb H_{\mathrm{pr}}\left(i_{\phi(t)}^k (I+Ti_{\phi(t)})^{-1}\tilde \eta_{(p)} \right)\right] \label{intr def section} 
\end{align}
Here, $\mathbb H_{\mathrm{pr}}$ denotes the harmonic projection onto the space of harmonic and primitive forms on $X_{t_0}$, $[\cdots]$ denotes the associated cohomology class, and $\tilde \eta_{(p)}$ is a basis of harmonic $(n-p,p)$-forms on $X_{t_0}$ representing the cohomology class $\eta_{(p)} = [\tilde \eta_{(p)}]$ at the base point.

By comparing the types of the summands in \eqref{intr lm section} and \eqref{intr def section}, we see that the sections $\Omega_{(p)}(t)$ and $\tO_{(p)}(t)$ of the Hodge bundles coincide on $t\in \tilde{S}^\vee$. Since $ \tO_{(p)}(t) $ is globally defined on ${\tilde{S}}$, the holomorphic section $\Omega_{(p)}(t) $ extends holomorphically to $ \tO_{(p)}(t) $ on ${\tilde{S}}$ by Riemann extension theorem. 

Now we deduce the following main result of this paper, which establishes a bridge between the algebraic description via the Lie theory of the period domain and the analytic construction from deformation theory.
\begin{theorem}\label{intr main bundle}
Let the notations be as above. Then the sections $\Omega_{(p)}(t)$ in~\eqref{intr lm section}, defined over $t \in \tilde S^{\vee}$ via the matrix representation of the image $\P(t)$ in $N_-$, can be holomorphically extended to the sections $\tO_{(p)}(t)$ in~\eqref{intr def section}, defined over $t \in \tilde S$ using the Beltrami differential $\phi(t)$ associated to the complex structure $X_{t} = (X_{t_{o}})_{\phi(t)}$. In particular, the sections $\tO_{(p)}(t)$ are holomorphic on $\tilde S$.
\end{theorem}

It is interesting to observe that the local property of being holomorphic for the sections of the Hodge bundles above arises from $ \Omega_{(p)}(t) $, which is given by the image of the period map, while the global existence of the sections comes from $ \tO_{(p)}(t) $, which is provided by deformation theory. 

Theorem~\ref{intr main bundle} implies that the limits of the blocks $\Phi^{(p,p+k)}(t)$ for $0 \le p \le n$ and $1 \le k \le n - p$ in~\eqref{intr lm section} exist as $t$ approaches points in the analytic subvariety ${\tilde{S}} \setminus \tilde{S}^{\vee}$. By the definition of $\tilde{S}^{\vee}$, this implies that ${\tilde{S}} = \tilde{S}^{\vee}$.
Hence, we obtain the following theorem, which provides a partial solution to Griffiths' Conjecture~10.1 in~\cite{Griffiths4}.

\begin{thm}\label{inmain}
Given an analytic family $f:\, \X \to S$ of polarized manifolds, the image of the lifted period map $\Phi:\, {\tilde{S}} \to D$ to the universal cover ${\tilde{S}}$ of $S$ lies in the complex Euclidean space $N_-$. 
\end{thm}

Theorem~\ref{inmain} shows that, after passing to the universal cover, the period map admits a global Euclidean realization inside the period domain. 
This provides a foundation for studying asymptotic behavior and compactifications of period maps.

Building on this Euclidean description and the underlying Hodge-theoretic constructions, we will prove that there exist affine structures on the Teichm\"uller spaces of Calabi--Yau type manifolds.

More precisely, we consider the moduli space of certain polarized manifolds with level $m$ structure, and let $\Z$ be one of the irreducible components of the moduli space. Then the Teichmüller space $\T$ is defined as the universal cover of $\Z$. Let $m_0$ be a positive integer. As a technical assumption, we will require that $\Z$ is smooth and carries an analytic family 
$$f_m:\, \U_m \to \Z$$ 
of polarized manifolds with level $m$ structure for all $m \geq m_0$. It is easy to show that $\T$ is independent of the levels as given in Lemma \ref{independent of m}.

\begin{thm}\label{affine-intro}
Let $X$ be a Calabi--Yau type manifold.
Assume that the Teichm\"uller space $\T$ of Calabi--Yau type manifolds containing $X$ exists as in Section \ref{moduli and period}.
Then there exists a global complex affine structure on $\T$.
\end{thm}

Here, the term {global complex affine structure} refers to the existence of a globally defined, non-degenerate holomorphic map
$$
\Psi: \T \to \C^{\dim_\C \T}.
$$

We illustrate Theorem~\ref{affine-intro} in the setting of Calabi--Yau manifolds. Let
\begin{equation}\label{intr pm CY type}
  \Phi:\, \T \to N_- \cap D,\, t\mapsto \left(\Phi^{(p,q)}(t)\right)
\end{equation}
denote the period map from the Teichmüller space $\T$ of Calabi--Yau manifolds. 

From Griffiths transversality and the infinitesimal Torelli theorem for Calabi--Yau manifolds, it follows that the Jacobian of the $(0,1)$-block,
$$
\frac{\partial\Phi^{(0,1)}}{\partial t_\mu}(t),
$$
which represents the first component
$$
\mathrm{Hom}\left(F_t^{n},\, F_t^{n-1}/F_t^{n}\right)
$$
of the differential of the period map at $t$, is non-degenerate, where $(t_\mu)$ denotes a holomorphic coordinate system around a point $t \in \T$. 
Consequently, the holomorphic  map
$$
\Psi: \T \to \C^{\dim_\C \T},\quad t\mapsto \Phi^{(0,1)}(t)
$$
defines a global complex affine structure on $\T$.

This paper is the first in a two-part series. In the subsequent paper~\cite{LS26-2}, we develop a non-polarized counterpart of this theory. 
While the present work focuses on global aspects of period maps for polarized families, Part~II treats the intrinsic theory for families of compact K\"ahler manifolds, introducing intrinsic period maps and Hodge maps for $(p,p)$-classes, together with applications to K\"ahler cones, algebraic approximation, and the variational Hodge conjecture. 
Taken together, the two papers suggest a unified framework in which the Beltrami differential, together with the Hodge theory of the central fiber, emerges as the fundamental structure governing both global and local variations of Hodge structures.

This paper is organized as follows. In Section~\ref{Lie}, we study period domains and period maps from a Lie-theoretic viewpoint, with particular attention to the unipotent orbit $N_-$ and the open subset ${\tilde{S}}^{\vee}$. Sections~\ref{basic lemma} and~\ref{Section-extensionforms} develop the construction of sections of Hodge bundles via the period map, through the matrix representation of its image in $N_-$, and via deformation theory, using Beltrami differentials, harmonic theory, and a closed extension formula. In Section~\ref{main proof}, we apply these constructions to establish a partial result toward Griffiths’ Conjecture~10.1. 
Section~\ref{affine and examples section} reviews the notions of moduli spaces and Teichm\"uller spaces, and constructs a global complex affine structure on the Teichm\"uller space of Calabi--Yau type manifolds.



\section{Period domains, period maps and unipotent orbits in period domains}\label{Lie}
In this section we review the definitions and basic properties of period domains  and period maps from Lie theory point of views. We also introduce the unipotent orbit in the period domain, which is key to the solution of Griffiths conjecture.

Let $H_{\mathbb{Z}}$ be a fixed lattice and $H=H_{\mathbb{Z}}\otimes_{\mathbb{Z}}
\C$ its  complexification. Let $n$ be a positive integer, and $Q$ a
bilinear form on $H_{\mathbb{Z}}$ which is symmetric if $n$ is even
and skew-symmetric if $n$ is odd. Let $h^{i,n-i}$, $0\le i\le n$, be
integers such that $$\sum_{i=0}^{n}h^{i,n-i}=\dim_{\C}H.$$ The period
domain $D$ for the polarized Hodge structures of type
$$\{H_{\mathbb{Z}}, Q, h^{i,n-i}\}$$ is the set of all the collections
of the subspaces $H^{i,n-i}$, $0\le i\le n$, of $H$ such that
$$H=\bigoplus_{0\le i\le n}H^{i,n-i}, \,H^{i,n-i}=\bar{H^{n-i,i}}, \, \dim_{\C} H^{i,n-i}= h^{i,n-i} \text{ for } 0\le i\le n,$$
and on which $Q$ satisfies the  Hodge-Riemann bilinear relations,
\begin{eqnarray}
Q\left ( H^{i,n-i}, H^{j,n-j}\right )=0\text{ unless }i+j=n;\label{HR1}\\
\left (\sqrt{-1}\right )^{2k-n}Q\left ( v,\bar v\right )>0\text{ for
}v\in H^{k,n-k}\setminus\{0\}. \label{HR2}
\end{eqnarray}

Alternatively, in terms of Hodge filtrations, the period domain $D$
is the set of all the  collections of the filtrations
$$H=F^{0}\supset F^{1}\supset \cdots \supset F^{n},$$ such that
\begin{align}
& \dim_{\mathbb{C}} F^i=f^i,  \label{periodcondition} \\
& H=F^{i}\oplus \bar{F^{n-i+1}},\text{ for } 0\le i\le n,\nonumber
\end{align}
where $f^{i}= h^{n,0}+\cdots +h^{i,n-i}$, and on which $Q$ satisfies the Hodge-Riemann bilinear relations in the form of Hodge filtrations
\begin{align}
& Q\left ( F^i,F^{n-i+1}\right )=0;\label{HR1'}\\
& Q\left ( Cv,\bar v\right )>0\text{ if }v\ne 0,\label{HR2'}
\end{align}
where $C$ is the Weil operator given by $$Cv=\left (\sqrt{-1}\right )^{2k-n}v$$ for $v\in F^{k}\cap \bar{F^{n-k}}$.

\begin{definition}
A polarized manifold is a pair $(X,L)$, where $X$ is a projective manifold and $L$ is an ample line bundle on $X$. Sometimes, we simply denote $X$ as a polarized manifold if the ample line bundle 
$L$ on it is obvious.
\end{definition}

From classical Hodge theory, we know that there exists a polarized Hodge structure of weight $k$
on the primitive cohomology group $H_{\mathrm{pr}}^k(X,\C)$ for $0\le k\le 2\dim_\C X$. 

An analytic family of polarized manifolds is a proper morphism $f :\, \X\to S$ between complex analytic spaces with the following properties
\begin{itemize}
\item[(1)] the complex analytic spaces $\X$ and $S$ are smooth and connected, and the morphism $f$ is non-degenerate, i.e. the tangent map $df$ is of maximal rank at each point of $\X$;
\item[(2)] there is a line bundle $\mathcal{L}$ over $\X$;
\item[(3)] $(X_s=f^{-1}(s),L_s=\mathcal{L}|_{X_s})$ is a polarized manifold for any $s\in S$.
\end{itemize}

Let ${\tilde{S}}$ be the universal cover of $S$. Then we have an induced family $f_T:\, \X_T \to{\tilde{S}}$ of polarized manifolds with an induced line bundle $\mathcal{L}_T$ on $\X_T$. 

Fixing a base point
$s_o\in S$ and $H_\mathbb{Z}=\H^n(X_{s_o},\mathbb{Z})/\mathrm{Tor}$, the period map is defined as a morphism $$\Phi :\, S \to \Gamma \backslash D$$ by
\begin{equation}\label{perioddefinition}
s \mapsto \tau^{[\gamma_s]}(F^n_s\subseteq \cdots\subseteq F^0_s)\in D,
\end{equation}
where $F_s^k=F^k \H^n(X_s,\mathbb{C})$ and $\tau^{[\gamma_s]}$ is an isomorphism between $\C-$vector spaces
$$\tau^{[\gamma_s]}:\, \H^n(X_s,\mathbb{C})\to H\simeq \H^n(X_{s_o},\mathbb{C}),$$
which depends only on the homotopy class $[\gamma_s]$ of the curve $\gamma_s$ between $s$ and $s_o$.

The period map from $S$ is well-defined with respect to the monodromy representation
$$\rho :\,  \pi_1(S,s_o)\to \Gamma \subseteq \text{Aut}(H_{\mathbb{Z}},Q),\, [\gamma]\mapsto \tau^{[\gamma]}.$$
It is well-known that the period map has the following properties (c.f. \cite{Griffiths1, Griffiths2}):
\begin{enumerate}
\item locally liftable;
\item holomorphic, i.e. $\partial F^i_z/\partial \bar{z}\subseteq  F^i_z$, $0\le i\le n$;
\item Griffiths transversality: $\partial F^i_z/\partial z\subseteq  F^{i-1}_z$, $1\le i\le n$.
\end{enumerate}
\begin{remark}
In this paper, we will also consider the period map from the base $S$ of analytic family of compact K\"ahler manifolds to the period domain $D$ of Hodge structures (not necessarily polarized) on $\H^n(X_s,\mathbb{C})$, $s\in S$. From the standard results of K\"ahler geometry, e.g. Section 2 in \cite{schmid1}, this period map is determined by the period maps to the period domains of polarized Hodge structures introduced as above.
\end{remark}

By passing to a finite-index, torsion-free subgroup of $\Gamma$, we may assume that $\Gamma$ itself is torsion-free, ensuring that $\Gamma \backslash D$ is smooth. Consequently, we can work on a finite cover of $S$ without loss of generality. For a standard construction of this type, we refer the reader to the proof of Lemma IV-A on pages 705--706 of \cite{Sommese}.

Since period map is locally liftable, we can lift the period map to $\P : {\tilde{S}} \to D$ by taking the universal cover ${\tilde{S}}$ of $S$ such that the diagram
\begin{equation*}
\xymatrix{
{\tilde{S}} \ar[r]^-{\P} \ar[d]^-{\pi} & D\ar[d]^-{\pi}\\
S \ar[r]^-{\Phi} & \Gamma \backslash D
}
\end{equation*}
is commutative.

Fix a base point $z_o\in \tilde S$ such that $\pi(z_o)=s_o$ is the base point in $S$.
For any $z=(X_z,L_z)\in{\tilde{S}}$, 
We denote 
the corresponding Hodge decomposition and Hodge filtration respectively by
\begin{eqnarray}
  &&\H^n(X_z, \mathbb{C})=\bigoplus_{p+q=n}  H_{\mathrm{pr}}^{p, q}(X_z) \label{givenHdgD}, \\
  &&\H^n(X _z, {\mathbb{C}})=F_z^{0}\supset F_z^{1}\supset\cdots \supset F_z^{n}, \label{givenHdgF}
\end{eqnarray}
where $$F_z^{p}= F^{p} \H^n(X _z, {\mathbb{C}})= H_{\mathrm{pr}}^{n,0}(X_z) \oplus \cdots \oplus H_{\mathrm{pr}}^{p,n-p}(X_z)$$ for $0\le p\le n$.

In this paper, all the vectors are taken as column vectors, since we will consider the left actions of $G_\C$ on the period domains. To simplify notations, we also write row of vectors  and use the transpose sign $T$ to denote the corresponding column.

Let us introduce the notion of adapted basis of the given Hodge decomposition or Hodge filtration.
We call a basis $$\xi=\left\{ \xi_0, \cdots, \xi_{f^{n}-1},\xi_{f^{n}}, \cdots ,\xi_{f^{n-1}-1}, \cdots, \xi_{f^{k+1}}, \cdots, \xi_{f^k-1}, \cdots, \xi_{f^{1}},\cdots , \xi_{f^{0}-1} \right\}^T$$ of $\H^n(X_z, \mathbb{C})$ an adapted basis
 for the given Hodge decomposition \eqref{givenHdgD}
  if it satisfies
 $$H^{k, n-k}(X_z)=\text{Span}_{\mathbb{C}}\left\{\xi_{f^{k+1}}, \cdots, \xi_{f^k-1}\right\},\, 0\le k\le n.$$
We call a basis
\begin{align*}
\zeta=\left\{ \zeta_0, \cdots, \zeta_{f^{n}-1},\zeta_{f^{n}}, \cdots ,\zeta_{f^{n-1}-1}, \cdots, \zeta_{f^{k+1}}, \cdots, \zeta_{f^k-1}, \cdots,\zeta_{f^{1}},\cdots , \zeta_{f^{0}-1} \right\}^T
\end{align*}
of $\H^n(X_z, {\mathbb{C}})$ an adapted basis of the given filtration \eqref{givenHdgF}
if it satisfies $$F^{k}_z=\text{Span}_{\mathbb{C}}\{\zeta_0, \cdots, \zeta_{f^k-1}\},\, 0\le k\le n.$$
For convenience, we set $f^{n+1}=0$ and $m=f^0$.

\begin{remark}
The adapted basis at the base point $z_0$ can be chosen either with respect to the given Hodge decomposition or the given Hodge filtration. Using methods from deformation theory, it is often convenient to select the adapted basis at $ z_0 $ in terms of harmonic forms. However, to ensure the holomorphicity of the period map, the adapted basis at any other point $ z $, given by the image of the period map, can only be chosen with respect to the Hodge filtration at $ z $.
\end{remark}

\begin{definition}\label{blocks}
(1) Let $$\xi=\left\{ \xi_0, \cdots, \xi_{f^{n}-1}, \cdots , \xi_{f^{k+1}}, \cdots, \xi_{f^k-1}, \cdots, \xi_{f^{1}},\cdots , \xi_{f^{0}-1} \right\}^T$$ 
be the adapted basis with respect to the Hodge decomposition or the Hodge filtration at any point.
The blocks of $\xi$ are defined by
$$\xi_{(\alpha)}=\{\xi_{f^{-\alpha +n+1}}, \cdots, \xi_{f^{-\alpha +n}-1}\}^T,$$
for $0\le \alpha \le n$. Then 
$$\xi=\{\xi_{(0)}^T,\cdots, \xi_{(n)}^T\}^T=\left(\begin{array}[c]{ccc}\xi_{(0)}\\\vdots\\ \xi_{(n)}\end{array}\right).$$

(2) The blocks of an $m\times m$ matrix $\Psi=(\Psi_{ij})_{0\le i,j\le m-1}$ are set as follows. For each
$0\leq \alpha, \beta\leq n$, the $(\alpha, \beta)$-th block
$\Psi^{(\alpha, \beta)}$ is defined by
\begin{align}\label{block}
\Psi^{(\alpha, \beta)}=\left(\Psi_{ij}\right)_{f^{-\alpha+n+1}\leq i \leq f^{-\alpha+n}-1, \ f^{-\beta+n+1}\leq j\leq f^{-\beta+n}-1}.
\end{align}
In particular, $\Psi =(\Psi^{(\alpha,\beta)})_{0\le \alpha,\beta \le n}$ is called a {block upper (lower resp.) triangular matrix} if
$\Psi^{(\alpha,\beta)}=0$ whenever $\alpha>\beta$ ($\alpha<\beta$ resp.).
\end{definition}

Let $H_{\mathbb{F}}=\H^n(X, \mathbb{F})$, where $\mathbb{F}$ is $\mathbb{Z}$, $\mathbb{R}$ or $\mathbb{C}$. Then $H=H_{\mathbb{C}}$ under this notation. We define the complex Lie group
\begin{align*}
G_{\mathbb{C}}=\{ g\in GL(H_{\mathbb{C}})|~ Q(gu, gv)=Q(u, v) \text{ for all } u, v\in H_{\mathbb{C}}\},
\end{align*}
and the corresponding real Lie group 
\begin{align*}
G_{\mathbb{R}}=\{ g\in GL(H_{\mathbb{R}})|~ Q(gu, gv)=Q(u, v) \text{ for all } u, v\in H_{\mathbb{R}}\}.
\end{align*}
We also have $$G_{\mathbb Z}=\text{Aut}(H_{\mathbb{Z}},Q) =\{ g\in \text{Aut}(H_{\mathbb{Z}})|~ Q(gu, gv)=Q(u, v) \text{ for all } u, v\in H_{\mathbb{Z}}\}.$$

Recall that $Q$ is a non-degenerate bilinear form on $H_{\mathbb{Z}}$ which is symmetric if $n$ is even, and skew-symmetric if $n$ is odd. Hence by choosing the standard basis of $H$ with respect to $Q$, we can show that the complex Lie subgroup $G_\C$ is the classical group $SO(N,\C)$ when the weight $n$ is even and $Sp(N,\C)$ when the weight $n$ is odd. Please see Chapter 4.4 of \cite{CMP} for the details.

There is a left action of $G_\C$ on $\check{D}$ such that the period domain $D$ can be realized as the $G_{\mathbb{R}}$-orbit.
Griffiths in \cite{Griffiths1} showed that $G_{\mathbb{C}}$ acts on
$\check{D}$ transitively, so does $G_{\mathbb{R}}$ on $D$. The
stabilizer of $G_{\mathbb{C}}$ on $\check{D}$ at the base point $$o=\left(H=\bigoplus_{p+q=n}H^{p,q}_o=F_o^0\supset \cdots \supset F^n_o\right)\in D$$
is a parabolic subgroup
$$B=\{g\in G_\C| gF_o^k=F_o^k,\ 0\le k\le n\},$$
and the stabilizer of $G_{\mathbb{R}}$ on $D$ is a compact real subgroup $V=B\cap G_\mathbb{R}$.
Thus we can realize $\check{D}$ and $D$ as
$$\check{D}=G_\C/B,\text{ and }D=G_\mathbb{R}/V,$$
so that $\check{D}$ is a flag manifold and $D\subseteq
\check{D}$ is an open complex submanifold, whihc is called flag domain.

The Lie algebra $\mathfrak{g}$ of the complex Lie group $G_{\mathbb{C}}$ is
\begin{align*}
\mathfrak{g}&=\{X\in \text{End}(H_\mathbb{C})|~ Q(Xu, v)+Q(u, Xv)=0, \text{ for all } u, v\in H_\mathbb{C}\},
\end{align*}
and the real subalgebra
$$\mathfrak{g}_0=\{X\in \mathfrak{g}|~ XH_{\mathbb{R}}\subseteq H_\mathbb{R}\}$$
is the Lie algebra of $G_\mathbb{R}$. Note that $\mathfrak{g}$ is a
simple complex Lie algebra and contains $\mathfrak{g}_0$ as a real
form,  i.e. $\mathfrak{g}=\mathfrak{g}_0\oplus \i\mathfrak{g}_0$.

On the vector space $\text{Hom}(H_\C,H_\C)$ we can give a Hodge structure of weight zero by
\begin{align*}
\mathfrak{g}=\bigoplus_{k\in \mathbb{Z}} \mathfrak{g}^{k,
-k},\quad\mathfrak{g}^{k, -k}= \{X\in
\mathfrak{g}:\,XH^{p,q}_o\subseteq H^{p+k, q-k}_o,\ \forall \, p+q=n
\}.
\end{align*}

By the definition of $B$, the Lie algebra $\mathfrak{b}$ of $B$ has
the  form $\mathfrak{b}=\bigoplus_{k\geq 0} \mathfrak{g}^{k, -k}$.
Then the Lie algebra $\mathfrak{v}_0$ of $V$ is
$$\mathfrak{v}_0=\mathfrak{g}_0\cap \mathfrak{b}=\mathfrak{g}_0\cap \mathfrak{b} \cap\bar{\mathfrak{b}}=\mathfrak{g}_0\cap \mathfrak{g}^{0, 0}.$$
With the above isomorphisms, the holomorphic tangent space of $\check{D}$ at the base point is naturally isomorphic to $\mathfrak{g}/\mathfrak{b}$.

Let us consider the nilpotent Lie subalgebra
$$\mathfrak{n}_-:=\bigoplus_{k\geq 1}\mathfrak{g}^{-k,k}.$$  Then one
has the isomorphism $\mathfrak{g}/\mathfrak{b}\cong
\mathfrak{n}_-$. We denote the corresponding unipotent Lie group to be
$$N_-=\exp(\mathfrak{n}_-).$$

\begin{remark}
We remark that the elements in $N_-$ can be realized as nonsingular block upper triangular matrices with identity blocks in the diagonal; the elements in $B$ can be realized as nonsingular block lower triangular matrices.
If $c, c'\in N_-$ such that $cB=c'B$ in $\check{D}$, then $$c'^{-1}c\in N_-\cap B=\{I \},$$ i.e. $c=c'$. This means that the matrix representation in $N_-$ of the unipotent orbit $N_-(o)$ is unique. Therefore with the fixed base point $o\in \check{D}$, we can identify $N_-$ with its unipotent orbit $N_-(o)$ in $\check{D}$ by identifying an element $c\in N_-$ with $[c]=cB$ in $\check{D}$. Therefore our notation $N_-\subseteq\check{D}$ is well-defined. In particular, when the base point $o$ is in $D$, we have $N_-\cap D\subseteq D$ is non-empty.

\end{remark}

As $\text{Ad}(g)(\mathfrak{g}^{k, -k})$ is in  $\bigoplus_{i\geq
k}\mathfrak{g}^{i, -i} \text{ for each } g\in B,$ the subspace
$\mathfrak{b}\oplus \mathfrak{g}^{-1, 1}/\mathfrak{b}\subseteq
\mathfrak{g}/\mathfrak{b}$ defines an Ad$(B)$-invariant subspace. By
left translation via $G_{\mathbb{C}}$,
$\mathfrak{b}\oplus\mathfrak{g}^{-1,1}/\mathfrak{b}$ gives rise to a
$G_{\mathbb{C}}$-invariant holomorphic subbundle of the holomorphic
tangent bundle, which will be denoted by $\mathrm{T}^{1,0}_{h}\check{D}$
and called as the horizontal tangent subbundle. One can
check that this construction does not depend on the choice of the
base point.

The horizontal tangent subbundle, restricted to $D$, determines a subbundle $\mathrm{T}_{h}^{1, 0}D$ of the holomorphic tangent bundle $\mathrm{T}^{1, 0}D$ of $D$.
The $G_{\mathbb{C}}$-invariance of $\mathrm{T}^{1, 0}_{h}\check{D}$ implies the $G_{\mathbb{R}}$-invariance of $\mathrm{T}^{1, 0}_{h}D$. Note that the horizontal
tangent subbundle $\mathrm{T}_{h}^{1, 0}D$ can also be constructed as the associated bundle of the principle bundle $V\to G_\mathbb{R} \to D$ with the adjoint
representation of $V$ on the space $\mathfrak{g}^{-1,1}$.

\begin{remark}\label{globally horizontal}
The $G_{\mathbb{C}}$-invariance of $\mathrm{T}^{1,0}_{h}\check{D}$ is of great importance to the study the Griffiths transversality in $N_-$. 
We also note that the subspace $$ \mathfrak{p}_{\mathfrak b}:\,=\left( \mathfrak{b}\bigoplus \bigoplus_{k>0, \, k \text{ odd}}\mathfrak{g}^{-k, k}\right)\Bigg/\mathfrak{b}\subseteq
\mathfrak{g}/\mathfrak{b}$$
shares many similar properties to the subspace $\mathfrak{b}\oplus \mathfrak{g}^{-1, 1}/\mathfrak{b}$ such as the curvature property. 
But $\mathfrak{p}_{\mathfrak b}$ is not preserved by the adjoint action of the parabolic subgroup $B$, and hence can not be defined globally on $\check D$ as $G_\C$-invariant subbundle.
\end{remark}

Let $$\mathcal{F}^{n}\subset \cdots\subset\mathcal{F}^{k}\subset \cdots \subset \mathcal{F}^{0}=\mathcal H=\check D \times H$$ be the  Hodge bundles on $\check D$
with fibers $\mathcal{F}^{k}|_{x}=F_{x}^{k}$ for any $x=(F_x^n\subset \cdots \subset F^0_x)\in \check D$, $0\le k\le n$. As another
interpretation of the horizontal bundle in terms of the Hodge
bundles $\mathcal{F}^{k}\to \check{D}$, $0\le k \le n$, one has
\begin{align}\label{horizontal}
\mathrm{T}^{1, 0}_{h}\check{D}\simeq \mathrm{T}^{1, 0}\check{D}\cap \bigoplus_{k=1}^{n}\text{Hom}(\mathcal{F}^{k}/\mathcal{F}^{k+1}, \mathcal{F}^{k-1}/\mathcal{F}^{k}).
\end{align}

The period map $\Phi:\, {\tilde{S}}\to D$ induces the pull-back vector bundles 
$\Phi^*\mathcal{F}^{k}$, $0\le k \le n$, which is still denoted by $\mathcal{F}^{k}$ and called Hodge bundles on ${\tilde{S}}$, such that there exists a flat connection $\nabla$ on $\mathcal F^0=\mathcal H$, called the Gauss-Mannin connection, satisfying the Griffiths transversality that
\begin{equation}\label{GM GT}
  \nabla:\, \mathcal F^k\to \Omega_{{\tilde{S}}}^1\otimes \mathcal F^{k-1}.
\end{equation}
In fact a period map $\Phi:\,{\tilde{S}}\to D$ is equivalent to the existence of the Hodge bundles on $S$ satisfying the above conditions.

Let
$$\Phi:\, {\tilde{S}}\to D$$
be the lifted period map from the universal cover ${\tilde{S}}$ of $S$.
Now we define
\begin{align*}
{\tilde{S}}^{\vee}=\P^{-1}(N_-\cap D).
\end{align*}
We first prove that ${\tilde{S}}\backslash{\tilde{S}}^{\vee}$ is an analytic subvariety of ${\tilde{S}}$ with $\text{codim}_{\mathbb{C}}({\tilde{S}}\backslash{\tilde{S}}^{\vee})\geq 1$.

\begin{lemma}\label{transversal}Let $z_o\in{\tilde{S}}$ be the base point with $\P(z_o)=\{F^n_{z_o}\subseteq F^{n-1}_{z_o}\subseteq \cdots\subseteq F^0_{z_o}\}.$ Let $z\in {\tilde{S}}$ be any point with $\P(z)=\{F^n_{z}\subseteq F^{n-1}_{z}\subseteq \cdots\subseteq F^0_{z}\}$, then $\P(z)\in N_-$ if and only if $F^{k}_{z}$ is isomorphic to $F^k_{z_o}$ for all $0\leq k\leq n$.
\end{lemma}
\begin{proof}
We fix $\eta= \{\eta_{(0)}^T, \cdots, \eta_{(n)}^T\}^T$ as the adapted basis of the Hodge filtration $\{F^{n}_{z_o} \subseteq F^{n-1}_{z_o}\subseteq \cdots\subseteq F^0_{z_o}\}$ at the base point $z_o$.
For any $z\in {\tilde{S}}$, we choose an arbitrary adapted basis $\zeta(z)=\{\zeta_{(0)}^T(z), \cdots, \zeta_{(n)}^T(z)\}^T$ for the given Hodge filtration $\{F^{n}_{z}\subseteq F^{n-1}_{z}\subseteq\cdots\subseteq F^0_{z}\}$.  Let $$A(z)=(A^{(\alpha, \beta)}(z))_{0\le \alpha, \beta \le n}$$ be the transition matrix between the basis $\zeta(z)$ and the basis $\eta$ for the same vector space $H_\C$, where $A^{(\alpha, \beta)}(z)$ are the corresponding blocks.
Then $$ \zeta(z)=A(z) \cdot \eta.$$ 
Therefore $$\P(z)\in N_-=N_-B/B\subseteq \check{D}$$ if and only if its matrix representation $A(z)$ can be decomposed as $L(z) U(z)$, where $L(z)\in N_-$ is a nonsingular block upper triangular matrix with identities in the diagonal blocks, and $U(z)\in B$ is a nonsingular block lower triangular matrix.

By basic linear algebra, we know that $(A^{(\alpha, \beta)}(z))$ has such decomposition if and only if $$\det(A^{(\alpha, \beta)}(z))_{0\leq \alpha, \beta \leq k}\neq 0$$ for any $0\leq k\leq n$. In particular, we know that $(A^{(\alpha, \beta)}(z))_{0\leq \alpha, \beta\leq k}$ is the transition map between the basis of $F^k_{z_o}$ and that  of $F^k_{z}$. Therefore, $$\det((A^{(\alpha, \beta)}(z))_{0\leq \alpha, \beta\leq k})\neq 0$$ if and only if $F^k_{z}$ is isomorphic to $F^k_{z_o}$.
\end{proof}

\begin{proposition}\label{codimension}
The subset ${\tilde{S}}^{\vee}$ is an open complex submanifold in ${\tilde{S}}$, and ${\tilde{S}}\backslash {\tilde{S}}^{\vee}$ is an analytic subvariety of ${\tilde{S}}$ with $\text{codim}_{\mathbb{C}}({\tilde{S}}\backslash {\tilde{S}}^{\vee})\geq 1$.
\end{proposition}
\begin{proof}
From Lemma \ref{transversal}, one can see that $\check{D}\setminus N_-\subseteq \check{D}$ is defined as an analytic subvariety by the equations
\begin{align*}
\{z\in \check{D} : \det ((A^{(\alpha, \beta)}(z))_{0\leq \alpha, \beta\leq k})=0\text{ for some } 0\leq k\leq n\}.
\end{align*}
Therefore $N_-$ is dense in $\check{D}$, and that $\check{D}\setminus N_-$ is an analytic subvariety,
which is closed in $\check{D}$ and with $\text{codim}_{\mathbb{C}}(\check{D}\backslash N_-)\geq 1$.

We consider the period map  $\P:\,{\tilde{S}}\to \check{D}$ as a holomorphic map to $\check{D}$, then $${\tilde{S}}\setminus
{\tilde{S}}^{\vee}=\P^{-1}(\check{D}\setminus N_-)$$ is the preimage of $\check{D}\setminus N_-$ of the holomorphic map $\P$. Therefore ${\tilde{S}}\setminus
{\tilde{S}}^{\vee}$ is also an analytic subvariety and a closed set in ${\tilde{S}}$. Because ${\tilde{S}}$ is smooth and connected, ${\tilde{S}}$ is
irreducible. If $\dim_\C({\tilde{S}}\setminus {\tilde{S}}^{\vee})=\dim_\C{\tilde{S}}$, then ${\tilde{S}}\setminus {\tilde{S}}^{\vee}={\tilde{S}}$ and
${\tilde{S}}^{\vee}=\emptyset$, but this contradicts to the fact that the base point $z_o$ is in ${\tilde{S}}^{\vee}$. Thus we conclude that
$$\dim_\C({\tilde{S}}\setminus {\tilde{S}}^{\vee})<\dim_\C{\tilde{S}},$$and consequently $\text{codim}_{\mathbb{C}}({\tilde{S}}\backslash {\tilde{S}}^{\vee})\geq
1$.
\end{proof}


\section{Sections of Hodge bundles from Hodge theory}\label{basic lemma}
In this section we give the sections of Hodge bundle on ${\tilde{S}}^{\vee}$ using the matrices representations of the image of the period map in $N_-$.

Let $\P:\, {\tilde{S}}^{\vee} \to N_-\cap D$ be the restricted period map as defined in Section \ref{Lie}. In Definition \ref{blocks}, we have introduced the blocks of the adapted basis $\eta$  of the Hodge decomposition at the base point $t_o\in \tilde{S}^\vee$ as
$$\eta=\{\eta_{(0)}^T,\eta_{(1)}^T,\cdots,\eta_{(n)}^T\}^T,$$
where 
$$\eta_{(p)}=\{\eta_{f^{-p +n+1}}, \cdots, \eta_{f^{-p +n}-1}\}^T$$ 
is the basis of $H^{n-p,p}(X_{t_o})$ for $0\le p \le n$.

For any $t\in {\tilde{S}}^{\vee}$, we can choose the matrix representation of the image $\P(t)$ in $N_-$ by 
$$\P(t)=\left(\Phi^{(p,q)}(t)\right)_{0\le p,q\le n}\in N_-\cap D.$$
Then, by Definition \ref{blocks}, the matrix $\P(t)$ is a block upper triangular matrix of the form, 
\begin{align}\label{block of n+}
\P(t)=\left(\begin{array}[c]{ccccccc}
I & \Phi^{(0,1)}(t) & \Phi^{(0,2)}(t) & \cdots& \Phi^{(0,n-1)}(t) & \Phi^{(0,n)}(t)\\ 
O& I& \Phi^{(1,2)}(t)&\cdots &\Phi^{(1,n-1)}(t)& \Phi^{(1,n)}(t)\\ 
O& O&I &\cdots &\Phi^{(2,n-1)}(t)&\Phi^{(2,n)}(t) \\ 
\vdots &\vdots & \vdots& \ddots& \vdots&\vdots \\ 
O&O &O &\cdots & I & \Phi^{(n-1,n)}(t)\\
 O&O &O &\cdots &  O&I
\end{array}\right),
\end{align}
where, in the above notations, $O$ denotes zero block matrix and $I$ denotes identity block matrix. Using the matrix representation, we have the adapted basis 
$$\Omega(t)=\{\Omega_{(0)}(t)^T,\Omega_{(1)}(t)^T,\cdots,\Omega_{(n)}(t)^T\}^T,$$
of the Hodge filtration at $t\in {\tilde{S}}^{\vee}$ as
\begin{eqnarray}\label{lm section}
\Omega_{(p)}(t)&=&\sum_{i=0}^{n-p}\Phi^{(p,p+i)}(t)\cdot \eta_{(p+i)} \nonumber \\
&=&\eta_{(p)}+\sum_{i=1}^{n-p} \Phi^{(p,p+i)}(t)\cdot \eta_{(p+i)}.
\end{eqnarray}
where $\Omega_{(p)}(t)$, together with $\Omega_{(0)}(t),\cdots, \Omega_{(p-1)}(t)$,
gives a basis of the Hodge filtration 
$$F^{n-p}_t = F^{n-p} H_{\mathrm{pr}}^n(X_t,\C).$$

Note that, in equation \eqref{lm section}, the term $\Phi^{(p,p+i)}(t)\cdot \eta_{(p+i)}$ is an element in $H_{\mathrm{pr}}^{n-p-i,p+i}(X_{t_o})$.

Moreover, since the period map is holomorphic, $\Omega_{(p)}(t)$ gives a holomorphic section of the Hodge bundle $\mathcal F^{n-p}$ on ${\tilde{S}}^{\vee}$, for $0\le p \le n$.


Let $\{U;t=(t_1,\cdots,t_N)\}$ be any local holomorphic coordinate on $\tilde{S}^\vee$.
In the following, the derivatives of the blocks,
$$
\frac{\partial \Phi^{(p,q)}}{\partial t_\mu}(t), \text{ for } 0\le p,q \le n,\ 1\le \mu\le N,
$$
will denote the blocks of derivatives of its entries,
$$
\frac{\partial \Phi^{(p,q)}}{\partial t_\mu}(t)=\left(\frac{\partial \Phi_{ij}}{\partial t_\mu}(t)\right)_{\substack{f^{-p+n+1}\leq i \leq f^{-p+n}-1\\ f^{-q+n+1}\leq j\leq f^{-q+n}-1}}.
$$

\begin{lemma}\label{lm derivative lemma}
Let the notations be as above. Then we have that 
\begin{align}\label{lm derivative}
\frac{\partial\Phi^{(p, p+i)}}{\partial t_\mu}(t)= \frac{\partial\Phi^{(p,p+1)}}{\partial t_\mu}(t) \cdot \Phi^{(p+1, p+i)}(t),
\end{align}
for any $0\le  p <p+i\le n$ and $1\le \mu\le N$.
\end{lemma}
\begin{proof}
The main idea of the proof is to rewrite the Griffiths transversality in terms of the matrix representation of the image of the period map in $N_-\cap D$. The rest only uses basic linear algebra. 

Let us consider the adapted basis
$$\Omega(t)=\{\Omega_{(0)}(t)^T,\Omega_{(1)}(t)^T,\cdots,\Omega_{(n)}(t)^T\}^T,$$
of the Hodge filtration $F^{n-p}_t$ at $t\in {\tilde{S}}^{\vee}$ as given in \eqref{lm section}.

By Griffiths transversality,  especially the computations in Page 813 of \cite{Griffiths2}, we have that 
$$\frac{\partial \Omega_{(p)}}{\partial t_\mu}(t)$$
lies in $F_t^{n-p-1}$.
By the construction of the basis in \eqref{lm section}, we know that 
$F_t^{n-p-1}$ has basis 
$$\{\Omega_{(0)}(t)^T,\Omega_{(1)}(t)^T,\cdots,\Omega_{(p+1)}(t)^T\}^T,$$ which gives local holomorphic sections of the Hodge bundle $\mathcal{F}^{n-p-1}$. Therefore there exist matrices $A_{(0)}(t), A_{(1)}(t),\cdots, A_{(p+1)}(t)$ such that 
\begin{eqnarray}\label{1}
\frac{\partial \Omega_{(p)}}{\partial t_\mu}(t)&=&A_{(0)}(t) \cdot\Omega_{(0)}(t)+\cdots +A_{(p+1)}(t)\cdot\Omega_{(p+1)}(t) \nonumber \\
&\equiv&A_{(0)}(t)\cdot \Omega_{(0)}(t)+\cdots +A_{(p)}(t)\cdot \Omega_{(p)}(t) \text{ mod } \bigoplus_{\gamma\ge 1}H_\mathrm{pr}^{n-p-\gamma,p+\gamma}(X_{t_o})\\
&\equiv&A_{(0)}(t) \cdot \eta_{(0)}+ \left[  A_{(0)}(t)\Phi^{(0,1)}(t) +A_{(1)}(t)\right]\cdot\eta_{(1)} + \cdots \label{1'}\\
&&+\left[ A_{(0)}(t) \Phi^{(0, p)}(t)  +\cdots +A_{(p-1)}(t)\Phi^{(p-1,p)}(t) +A_{(p)}(t)\right]\cdot\eta_{(p)}   \nonumber \\
&& \text{ mod } \bigoplus_{\gamma\ge 1}H_\mathrm{pr}^{n-p-\gamma,p+\gamma}(X_{t_o}),\nonumber
\end{eqnarray}
where equations \eqref{1} and \eqref{1'} are deduced by comparing the types in \eqref{lm section} and the blocks in \eqref{block of n+}. Here we mod out certain Hodge types at the base point $t_o$.

Again, by looking at \eqref{lm section}, we have that 
\begin{eqnarray}\label{2}
\frac{\partial \Omega_{(p)}}{\partial t_\mu}(t)&=&\sum_{i\ge 1} \frac{\partial \Phi^{(p, p+i)}}{\partial t_\mu}(t)\cdot \eta_{(p+i)}\nonumber \\
&=& \frac{\partial\Phi^{(p,p+1)}}{\partial t_\mu}(t)\cdot \eta_{(p+1)}+\sum_{i\ge 2} \frac{\partial \Phi^{(p, p+i)}}{\partial t_\mu}(t) \cdot\eta_{(p+i)}\nonumber\\
&\equiv&0 \text{ mod } \bigoplus_{\gamma\ge 1}H_\mathrm{pr}^{n-p-\gamma,p+\gamma}(X_{t_o}).
\end{eqnarray}

By comparing equations \eqref{1'} and \eqref{2}, we see that 
\begin{eqnarray*}
\left\{ \begin{array}{lr} A_{(0)}(t) =0\\
A_{(0)}(t)\Phi^{(0,1)}(t) +A_{(1)}(t) = 0\\
\dots \dots\\
A_{(0)}(t) \Phi^{(0, p)}(t)  +\cdots +A_{(p-1)}(t)\Phi^{(p-1,p)}(t) +A_{(p)}(t)=0,
\end{array} \right.
\end{eqnarray*}
which implies inductively that $A_{(0)}(t)=O,\cdots, A_{(p)}(t)=O$. 
Therefore 
\begin{eqnarray*}
\frac{\partial \Omega_{(p)}}{\partial t_\mu}(t)&=&A_{(p+1)}(t)\Omega_{(p+1)}(t) \\
&=&A_{(p+1)}(t) \cdot\left(\eta_{(p+1)}+\sum_{i\ge 2}\Phi^{(p+1, p+i)}(t) \cdot\eta_{(p+i)}\right) \\
&=& A_{(p+1)}(t) \cdot \eta_{(p+1)} +\sum_{i\ge 2} A_{(p+1)}(t) \Phi^{(p+1, p+i)}(t) \cdot \eta_{(p+i)}.
\end{eqnarray*}

By comparing types again, we have that
\begin{align*}
A_{(p+1)}(t)= \frac{\partial\Phi^{(p,p+1)}}{\partial t_\mu}(t)
\end{align*}
and
\begin{eqnarray}
\frac{\partial\Phi^{(p, p+i)}}{\partial t_\mu}(t)&=&A_{(p+1)}(t) \cdot \Phi^{(p+1, p+i)}(t)  \label{lm derivative'}\\
&= &\frac{\partial\Phi^{(p,p+1)}}{\partial t_\mu}(t) \cdot \Phi^{(p+1, p+i)}(t),\nonumber 
\end{eqnarray}
for $i \ge 2$. Since $\Phi^{(p+1,p+1)}(t)$ is identity matrix,  \eqref{lm derivative'} holds  trivially for $i=1$. Therefore, we have proved equation \eqref{lm derivative}.
\end{proof}

\begin{remark}\label{trans remark}
Note that, by construction in \eqref{lm section},
each $\Omega_{(p)}(t)$ can be considered as a basis of $F_t^{n-p}/F_t^{n-p+1}$, $0\le p \le n$.

By the proof of Lemma \ref{lm derivative lemma}, we have that 
$$ (d\P)_{t}\left(\frac{\partial}{\partial t_{\mu}}\right)=\bigoplus_{0\le p\le n-1}\frac{\partial\Phi^{(p,p+1)}}{\partial t_\mu}(t)$$
as elements in 
$$\bigoplus_{0\le p\le n-1}\mathrm{Hom}(F_t^{n-p}/F_t^{n-p+1},F_t^{n-p-1}/F_t^{n-p})$$
which maps $\Omega_{(p)}(t)$ to $$A_{(p+1)}(t)\cdot \Omega_{(p+1)}(t)= \frac{\partial\Phi^{(p,p+1)}}{\partial t_\mu}(t)\cdot\Omega_{(p+1)}(t).$$
\end{remark}

\section{Sections of Hodge bundles from deformation theory}
\label{Section-extensionforms}
Recall that ${\tilde{S}}$ is a simply-connected complex manifold on which there is an analytic family $\tilde{\mathcal{X}} \to {\tilde{S}}$ of polarized manifolds with a line bundle $\tilde{\mathcal{L}}$ on $\tilde{\mathcal{X}}$.
In this section, we give the sections of Hodge bundles $\mathcal F^{\alpha}$, $0 \le \alpha \le n$, on ${\tilde{S}}$, utilizing the Beltrami differentials in deformation theory and the operators from the harmonic theory on compact K\"ahler manifolds.

In Section \ref{subsection-Beltrami} and Section \ref{section-Hodge}, we review the basics of Beltrami differentials and Harmonic theory on compact K\"ahler manifolds respectively.
Then in Section \ref{subsection-extension}, we present a closed formula for constructing extensions as $d$-closed forms of harmonic forms under global variation of complex structures. Finally, in Section \ref{gsec def}, we apply the closed formula to ${\tilde{S}}$ to get the global sections of Hodge bundles on ${\tilde{S}}$.

See \cite{LZ18} for further details on applications of the methods developed in this section.
 
\subsection{Beltrami differentials} \label{subsection-Beltrami}
\noindent

In this section, $X$ is a complex manifold with
$\dim_{\mathbb{C}}X=d$, and we denote by $M$ the underlying differential
manifold of $X$ of real dimension $2d$. The associated
complex structure of the complex manifold $X$ gives a direct sum
decomposition of the complexified differential tangent bundle $\mathrm{T}_{\mathbb{C}}M$,
\begin{align*}
\mathrm{T}_{\mathbb{C}}M=\Tan X\oplus \mathrm{T}^{0,1}X.
\end{align*}
Let $J$ be any almost complex structure on $M$. Then, $J$ gives
another direct sum decomposition,
\begin{align*}
\mathrm{T}_{\mathbb{C}}M=\mathrm{T}^{1,0}X_{J}\oplus \mathrm{T}^{0,1}X_J.
\end{align*}
Denote by
\begin{align*}
\iota_1: \mathrm{T}_{\mathbb{C}}M\rightarrow \mathrm{T}^{1,0}X, \  \iota_2:
\mathrm{T}_{\mathbb{C}}M\rightarrow \mathrm{T}^{0,1}X,
\end{align*}
the two projection maps onto the holomorphic and anti-holomorphic tangent bundles of the base complex manifold $X$.

\begin{definition} [cf. Definition 4.2 \cite{Kir} ]\label{finite dist}
Let $J$ be an almost complex structure on $M$, we say that $J$ is of
finite distance from the given complex structure $X$ on $M$, if the
restriction map
\begin{align*}
\iota_1|_{\mathrm{T}^{1,0}X_J}: \mathrm{T}^{1,0}X_J\rightarrow \mathrm{T}^{1,0}X
\end{align*}
is an isomorphism.
\end{definition}
Therefore, if $J$ is of finite distance from $X$, one can define a
map
\begin{align}\label{barphi defn}
\bar{\phi}: \mathrm{T}^{1,0}X\rightarrow \mathrm{T}^{0,1}X
\end{align}
by setting $$\bar{\phi}(v)=-\iota_2 \circ
\left(\iota_1|_{\mathrm{T}^{1,0}X_J}\right)^{-1}(v).$$This map is
well-defined since $\iota_1|_{\mathrm{T}^{1,0}X_J}$ is an isomorphism. 

The map $\bar\phi$ in \eqref{barphi defn}, or its complex conjugate 
\begin{align}\label{phi defn}
\phi: \, \mathrm{T}^{0,1}X \to \mathrm{T}^{1,0}X, \,v\mapsto {\phi(v)}=\bar{\bar{\phi}(\bar{v})}
\end{align}
completely determines the almost complex structure $J$. More precisely, we have
\begin{align}
\mathrm{T}^{1,0}X_J=\{v-\bar{\phi}(v):\,v\in \mathrm{T}^{1,0}X \},\
\mathrm{T}^{0,1}X_J=\{v-\phi(v):\,v\in \mathrm{T}^{0,1}X \},
\end{align}
and the corresponding dual spaces
\begin{align}\label{basis of cotang}
{\mathrm{T}^*}^{1,0}X_J=\{w+\phi(w):\,w\in {\mathrm{T}^*}^{1,0}X \},\
{\mathrm{T}^*}^{0,1}X_J=\{w+\bar{\phi}(w):\,w\in {\mathrm{T}^*}^{0,1}X \}.
\end{align}
The bundle morphism $\phi$ in \eqref{phi defn} gives a global section
$$\phi\in \Gamma(X,\left(\mathrm{T}^{0,1}X\right)^{*}\otimes\mathrm{T}^{1,0}X)=A^{0,1}(X,\mathrm{T}^{1,0}X).$$

Since
\begin{align*}
\mathrm{T}^{1,0}X\oplus \mathrm{T}^{0,1}X=\mathrm{T}_{\mathbb{C}}M=\mathrm{T}^{1,0}X_J\oplus \mathrm{T}^{0,1}X_J,
\end{align*}
the transformation matrix
\begin{equation*}
\left(
\begin{array}{ll}
I_n & -\bar{\phi}  \\
-\phi & I_n
\end{array}\right)
\end{equation*}
from a basis of $\mathrm{T}^{1,0}X\oplus \mathrm{T}^{0,1}X$ to a basis of
$\mathrm{T}^{1,0}X_J\oplus \mathrm{T}^{0,1}X_J$ must be nondegenerate. Therefore $
\det(I_n-\phi\bar{\phi})\neq 0. $  In fact, we have

\begin{proposition}[cf. Proposition 4.3 in \cite{Kir}] \label{propfinitedistance}
There is a bijective correspondence between the set of almost
complex structures of finite distance from $X$ and the set of all
$\phi \in A^{0,1}(X,\mathrm{T}^{1,0}X)$ such that, at each point $p \in M$,
the map
$$\phi \bar{\phi}:\, \mathrm{T}^{1,0}X \to \mathrm{T}^{1,0}X$$
has no eigenvalue equal to $1$.
\end{proposition}

\begin{definition}
If $\phi \in A^{0,1}(X,\mathrm{T}^{1,0}X)$ satisfies the condition in
Proposition \ref{propfinitedistance}, we say that $\phi$ is a
\emph{Beltrami differential}. If $\phi$ further satisfies the Maurer–Cartan
equation (i.e., the integrability condition)
\begin{align} \label{MCequ}
\bar{\partial} \phi = \frac{1}{2}[\phi,\phi],
\end{align}
then we call
$\phi$ an \emph{integrable Beltrami differential}.
\end{definition}

In conclusion, a Beltrami differential $\phi\in A^{0,1}(X,\mathrm{T}^{1,0}X)$ determines an
almost complex structure of finite distance from $X$. We denote the
corresponding almost complex structure (i.e. almost complex
manifold) by $X_\phi$. An integrable Beltrami differential
$\phi$ gives a complex structure  on $M$ by the
Newlander-Nirenberg theorem \cite{NN}, and the corresponding complex
manifold is denoted by $X_\phi$.

\subsection{Harmonic theory on compact K\"ahler manifolds}
\label{section-Hodge}
\noindent

In this section, we briefly review the Harmonic theory on compact
K\"ahler manifolds and fix the notations used in the next subsections.

Let $(E,h)$ be a Hermitian holomorphic vector bundle over a compact
complex manifold $M$ with Hermitian metric $g$. Let
$\nabla=\nabla'+\bp$ be the Chern connections of $(E,h)$.  The
Hermitian metrics on $E$ and $M$ induce an $L^2$ inner produce on
the space $A^{p,q}(M,E)$ of $E$-valued $(p,q)$-forms on $M$. We set
the Laplacians
\begin{align*}
\triangle_{\bp}=\bp \bp^*+\bp^*\bp \ \ \text{and} \ \
\triangle'=\nabla'\nabla'^*+\nabla'^*\nabla'.
\end{align*}
Hodge theory implies that there are  Green operator $G$ (resp. $G'$)
and harmonic projection $\mathbb{H}$ (resp. $\mathbb{H}'$) in the
Hodge decomposition corresponding to $\triangle_{\bp}$ (resp.
$\triangle^{'}$).
\begin{proposition} \label{Prop-Hodgeidentities}
We have the following identities:
\begin{align*}
\triangle_{\bar{\partial}}G=G\triangle_{\bar{\partial}}=I-\mathbb{H}, \
\bar{\partial}G=G\bar{\partial}, \ \bar{\partial}^*
G=G\bar{\partial}^*, \ \mathbb{H} G=G \mathbb{H}=0.
\end{align*}
Moreover, $ \bar{\partial} \mathbb{H}= \mathbb{H}\bar{\partial}=0, \
\bar{\partial}^* \mathbb{H}=\mathbb{H}\bar{\partial}^*=0. $  The
similar identities holds among the operators  $G'$, $\mathbb{H}'$,
$\nabla'$ and $\nabla'^*$.
\end{proposition}

The Green operator is a bounded operator on the compact complex manifold, that means if we associate
certain norm $\|\cdot\|$, such as the $C^{k,\alpha}$-norm or the $L^{2}$-norm, on
$A^{p,q}(M,E)$, there is a constant $C$ independent $\sigma$, such
that
\begin{align}
\|G\sigma\|\leq C \|\sigma\| \ \text{for any}  \ \sigma\in
A^{p,q}(M,E).
\end{align}

Then we suppose $(M,\omega)$ is an $n$-dimensional compact K\"ahler
manifold with K\"ahler metric $\omega$,  and $\|\cdot\|_{L^2}$ be
the $L^{2}$-norm on the space $A^{p,q}(M)$ of smooth differential
forms induced by the  metric $\omega$. We set
\begin{align*}
\triangle_{\bp}=\bp \bp^*+\bp^*\bp,  \ \ \triangle_{\p}=\p\p^*+\p^*\p \
\text{and} \ \Delta_d=dd^*+d^*d,
\end{align*}
where $d=\p+\bp$.  

On $A^{p,q}(M)$, we have the equality of the
Laplacians $
\triangle_{\bar{\partial}}=\triangle_\partial=\frac{1}{2}\Delta_d. $ We
also let $$\mathbb{H}:\,=\mathbb{H}_{\bp}=\mathbb{H}_\partial=\mathbb{H}_d$$ to be the orthogonal projection from $A^{p,q}(M)$ to the harmonic space 
$$\mathbb{H}^{p,q}(M)=\ker
\triangle_{\bar{\partial}}=\ker
\triangle_{{\partial}}=\ker
\triangle_{d}.$$ 
Then the corresponding identities in
Proposition \ref{Prop-Hodgeidentities} hold. Furthermore, since the Laplacians
$\triangle_{\bar{\partial}}=\triangle_\partial$ and the corresponding Green operators 
$G_{\bar{\partial}}=G_\partial:\,=G$ 
on compact K\"ahler
manifold, we can derive the following proposition. 

\begin{proposition}
Let the operator $T:\, A^{p,q}(M) \to A^{p,q}(M)$ be defined by $T=\bar{\partial}^*G\partial$.
Then, for any $g\in A^{p,q}(M)$, we have that the quasi-isometry formula
\begin{align} \label{quasi-iso}
\|Tg\|^{2}_{L^2}=\|g\|_{L^2}^2-\|\mathbb{H}g\|_{L^2}^2-\|\partial\partial^* Gg\|^{2}-\|\bar{\partial} G\partial g\|^2_{L^2}\leq \|g\|_{L^2}
\end{align}
\end{proposition}
\begin{proof} 
For any $g\in A^{p,q}(M)$, we have
\begin{align}
\|\bar{\partial}^*G\partial g\|_{L^2}^2&=\langle 
\bar{\partial}^*G\partial g,\bar{\partial}^*G\partial g\rangle
=\langle \bar{\partial}\bar{\partial}^*G\partial g, G\partial g\rangle\notag\\
&=\langle \triangle_{\bar{\partial}}G\partial g,G\partial g\rangle-\langle \bar{\partial}^*\bar{\partial}G\partial g, G\partial g\rangle\notag\\
&=\langle \partial g,G\partial g\rangle-\langle \bar{\partial}G\partial g,\bar{\partial}G\partial g\rangle\notag\\
&=\langle  g,\triangle_\partial Gg\rangle-\langle  g, \partial\partial^* Gg\rangle-\|\bar{\partial} G\partial g\|^2_{L^2}\notag\\
&=\langle  g, g-\mathbb{H}g\rangle-\langle  \partial^*g,G\partial^*g\rangle-\|\bar{\partial} G\partial g\|^2_{L^2}\notag\\
&=\|g\|_{L^2}^2-\|\mathbb{H}g\|_{L^2}^2-\langle  \partial^*g,G\partial^*g\rangle-\|\bar{\partial} G\partial g\|^2_{L^2}\notag.
\end{align}
According to the Harmonic theory, we have the orthogonal decomposition of $g$,
$$g=\mathbb Hg+\partial\partial^{*}Gg+\partial^{*}\partial Gg.$$
Hence $$\langle  \partial^*g,G\partial^*g\rangle = \langle  g,\partial\partial^* Gg\rangle =\|\partial\partial^* Gg\|^{2},$$
which implies \eqref{quasi-iso}.
\end{proof}
The {\em quasi-isometry} formula was first proved in \cite{LRY}.

\subsection{Extension equations and closed formulas} \label{subsection-extension}
\noindent

Let $X$ be a complex manifold of dimension $n$. Let $\phi\in
A^{0,k}(X,\mathrm{T}^{1,0}X)$ be a $\mathrm{T}^{1,0}X$-value $(0,k)$-form. We
introduce the contraction operator
\begin{align*}
i_{\phi}: A^{p,q}(X)\rightarrow A^{p-1,q+k}(X)
\end{align*}
as in \cite{LRY}. If we write $\phi=\eta\otimes Y$ with $\eta\in
A^{0,k}(X)$ and $Y\in C^{\infty}(\mathrm{T}^{1,0}X)$, then for $\sigma\in
A^{p,q}(X)$,
\begin{align*}
i_{\phi}(\sigma)=\eta\wedge i_Y\sigma.
\end{align*}
Sometimes, we also use the notations $\phi\lrcorner
\eta=\phi\eta$ to denote the contraction $i_{\phi} \eta$
alternatively. By definition, we have
\begin{align*}
i_{\phi}i_{\phi'}=(-1)^{(k+1)(k'+1)}i_{\phi'}i_{\phi}
\end{align*}
if $\phi\in A^{0,k}(X)$ and $\phi' \in A^{0,k'}(X)$. The Lie
derivation of $\mathcal{\phi}$ is defined by
\begin{align*}
\mathcal{L}_{\phi}=(-1)^kd\circ i_\phi+i_\phi \circ d
\end{align*}
which can be decomposed into the sum of two parts
\begin{align*}
\mathcal{L}_{\phi}^{1,0}=(-1)^k\partial \circ i_\phi
+i_\phi \circ \partial, \ \
\mathcal{L}_{\phi}^{0,1}=(-1)^k\bar{\partial}\circ i_\phi
+i_\phi \circ\bar{\partial}.
\end{align*}
The Lie bracket of $\phi$ and $\phi'$ is defined by
\begin{align*}
[\phi,\phi']=\sum_{i,j=1}^n\left(\phi^i \wedge
\partial_i \phi'^j-(-1)^{kk'}\phi'^i\wedge \partial_i
\phi^j\right)\otimes \partial_j,
\end{align*}
if  $\phi=\sum_{i}\phi^i\partial_{i}\in A^{0,k}(X,\mathrm{T}^{1,0}X)$
and $\phi'=\sum_{i}\phi'^i\partial_{i}\in
A^{0,k'}(X,\mathrm{T}^{1,0}X)$.

We have the following generalized Cartan formula \cite{LR,LRY} which
can be proved by direct computations.
\begin{lemma}\label{cartan}
For any $\phi, \phi'\in A^{0,1}(X,\mathrm{T}^{1,0}X)$, then on $
A^{*,*}(X)$ we have that
\begin{align} \label{genralizedcartan}
i_{[\phi,\phi']}=\mathcal{L}_{\phi}\circ
i_{\phi'}-i_{\phi'}\circ\mathcal{L}_{\phi}.
\end{align}
\end{lemma}

Let $\sigma\in A^{*,*}(X)$. By applying the formula
(\ref{genralizedcartan}) to $\sigma$ and considering the types,  we
immediately obtain
\begin{align} \label{genralizedcartan_special}
[\phi,\phi]\lrcorner \sigma=2\phi\lrcorner
\partial(\phi\lrcorner \sigma) -\partial(\phi\lrcorner
\phi\lrcorner\sigma)-\phi\lrcorner\phi\lrcorner
\partial\sigma.
\end{align}

Given a Beltrami differential $\phi$, for any $x\in X$, we can
pick a local holomorphic coordinate $(U,z_1,...,z_n)$ near $x$. Then
by Proposition 1.11 of \cite{Griffiths2}, we have that
\begin{align*}
{\mathrm{T}^{*}}^{1,0}_x(X_{\phi})=\text{Span}_{\mathbb{C}}\{dz_1+\phi
dz_1,...,dz_n+\phi dz_n\},
\end{align*}
and for any $1\leq p\leq n$,
\begin{align*}
\bigwedge^{p}\mathrm{T}^{1,0}_x(X_{\phi})=\text{Span}_{\mathbb{C}}\{(dz_{i_1}+\phi
dz_{i_1})\wedge\cdots \wedge (dz_{i_p}+\phi dz_{i_p})|1\leq
i_1<\cdots< i_p\leq n\}.
\end{align*}

Consider  the operator $\rho_{\phi}=e^{i_\phi}$ defined by 
\begin{align} \label{exponentialphi}
\rho_{\phi}(\sigma)= e^{i_\phi}\sigma= \sum_{k\geq
0}\frac{1}{k!}i_{\phi}^k\sigma,\, \forall\, \sigma \in A^{p,q}(X).
\end{align}
Let $\sigma=\sum_{I,J}\sigma_{I,\bar{J}}dz_{I}\wedge dz_{\bar J} \in A^{p,q}(X)$.
By a straightforward computation we obtain
\begin{align}
e^{i_{\phi}}(\sigma)&=\sum_{I,J}\sigma_{I,\bar{J}}e^{i_{\phi}}(dz_{i_1}\wedge\cdots
\wedge dz_{i_p})\wedge d\bar{z}^{j_1}\wedge\cdots \wedge
d\bar{z}^{j_q}\nonumber \\
&=\sum_{I,J}\sigma_{I,\bar{J}}(dz_{i_1}+\phi
dz_{i_1})\wedge\cdots\wedge (dz_{i_p}+\phi dz_{i_p})\wedge
d\bar{z}^{j_1}\wedge\cdots \wedge d\bar{z}^{j_q} ,\label{e-sigma}
\end{align}
where $\phi d\bar{z}^{j}=0$, for any $j$.

Since the total degree of $e^{i_{\phi}}(\sigma)$ is $p+q$, and \eqref{e-sigma} implies that $e^{i_{\phi}}(\sigma)$ contains at least $p$-holomorphic forms on $X_{\phi}$, we have that $e^{i_\phi}$ is a linear map $$e^{i_\phi}: A^{p,q}(X)\rightarrow
A^{p,q}(X_\phi)\oplus A^{p+1,q-1}(X_\phi)\oplus \cdots\oplus A^{p+q,0}(X_\phi):\,= F^pA^{p+q}(X_\phi),$$
which induces an isomorphism 
$$e^{i_\phi}:\, F^pA^{p+q}(X)\rightarrow F^pA^{p+q}(X_\phi).$$

By the generalized Cartan formula (\ref{genralizedcartan}), it
follows that
\begin{proposition} \cite{Clemens,LRY}
Let $\phi \in A^{0,1}(X,\mathrm T^{1,0}X)$, then on $A^{*,*}(X)$, we have
\begin{align} \label{commuted}
e^{-i_\phi}\circ d\circ
e^{i_\phi}=d-\mathcal{L}_{\phi}-i_{\frac{1}{2}[\phi,\phi]}.
\end{align}
\end{proposition}
\begin{proof}
Let $\eta\in A^{*,*}(X)$, formula (\ref{genralizedcartan}) implies
\begin{align} \label{genralizedcartan1}
d i_\phi^2\eta=2i_\phi  d i_\phi\eta-i_\phi^2
d\eta-i_{[\phi,\phi]}\eta.
\end{align}
Substituting $\eta$ with $i_\phi \eta$ in
(\ref{genralizedcartan1}), we obtain
\begin{align*}
d(i_\phi^3 \eta)&=2i_\phi d(i_\phi^2 \eta)-i_\phi^2
di_\phi \eta-i_{[\phi,\phi]}i_\phi\eta\\\nonumber
&=3i_\phi^2 d(i_\phi \eta)-2i_\phi^3d\eta-3i_\phi
i_{[\phi,\phi]}\eta.
\end{align*}
where we have used $i_\phi
i_{[\phi,\phi]}=i_{[\phi,\phi]}i_{\phi}$. Then by
induction, we immediately have
\begin{align*}
d(i_\phi^k \eta)=ki_\phi^{k-1}d(i_\phi
\eta)-(k-1)i_\phi^kd\eta-\frac{k(k-1)}{2}i_\phi^{k-2}i_{[\phi,\phi]}\eta
\end{align*}
for $k\geq 2$. Through a straightforward computation, it is easy to
show that
\begin{align*}
d(e^{i_\phi}\eta)=e^{i_\phi}\left(d\eta-\mathcal{L}_{\phi}\eta-i_{\frac{1}{2}[\phi,\phi]}\eta\right).
\end{align*}
The proof of (\ref{commuted}) is completed.
\end{proof}

\begin{corollary} \label{commuted1}
If $\phi\in A^{0,1}(X,\mathrm T^{1,0}X)$ is integrable, then on
$A^{*,*}(X)$, we have
\begin{align} \label{obstructionequation1}
e^{-i_\phi}\circ d\circ
e^{i_\phi}=d-\mathcal{L}_{\phi}^{1,0}=d+\partial
i_\phi-i_{\phi} \partial.
\end{align}
\end{corollary}
\begin{proof}
By a straightforward computation, we have $
\bar{\partial}i_{\phi}-i_{\phi}\bar{\partial}=i_{\bar{\partial}\phi}.
$ Then Corollary \ref{commuted1} follows directly from the
integrability of $\phi$, i.e. $\bp
{\phi}=\frac{1}{2}[\phi,\phi]$.
\end{proof}
In particular, we have

\begin{corollary} \label{coro-obstructionequation2}
Let $\phi\in A^{0,1}(X,\mathrm T^{1,0}X)$ be an integrable Beltrami differential. Then for any smooth form $\sigma \in A^{p,q}(X)$, the
corresponding form
$\rho_{\phi}(\sigma)=e^{i_{\phi}}(\sigma)$ in $F^pA^{p+q}(X_\phi)$ is
$d$-closed, if and only if
\begin{equation}\label{obstructionequation2}
\left\{ \begin{array}{lr} {\partial}\sigma=0,\\
\bar{\partial}\sigma =-\partial(\phi\lrcorner \sigma).
\end{array} \right.
\end{equation}
\end{corollary}
\begin{proof}
From Corollary \ref{commuted1}, we have that $e^{i_\phi}(\sigma)$ is $d$-closed if
and only if
$$\partial\sigma+\bp\sigma +\partial(\phi\lrcorner \sigma)-\phi\lrcorner (\partial\sigma)=0.$$
By comparing types, we have \eqref{obstructionequation2}.
\end{proof}

\begin{definition}\label{son defn}
The supremum operator norm of a Beltrami differential 
$\phi \in A^{0,1}(X,\mathrm T^{1,0}X)$ is defined as
\begin{equation}\label{supremum operator norm}
\|\phi\|^{\omega} := \sup_{x\in X} \|\phi\|^{\omega}_{2,x},
\end{equation}
where the pointwise operator norm is given by
\begin{equation*}
\|\phi\|^{\omega}_{2,x} := \sup_{v\in \mathrm T_{x}^{0,1}X}
\frac{\|\phi(v)\|_{x}^{\omega}}{\|v\|_{x}^{\omega}}, \quad x\in X,
\end{equation*}
and $\|\cdot\|_{x}^{\omega}$ denotes the Hermitian metric on 
$\mathrm T_{x}^{1,0}X$ and $\mathrm T_{x}^{0,1}X$ induced by the K\"ahler form $\omega$.\end{definition}

Since $X$ is compact, there exists some $x_{0}\in X$ such that $\|\phi\|=\|\phi\|_{2,x_{0}}$.

%

We note that the symbol $\|\cdot\|$ is also used to denote the $L^2$ norm on $A^{p,q}(X)$ in this paper. Here, however, $\|\phi\|$ will always refer to the supremum operator norm of a Beltrami differential $\phi \in A^{0,1}(X,\mathrm T^{1,0}X)$ with respect to the K\"ahler metric $\omega$, and thus no ambiguity will arise.

We are ready to prove the following result.

\begin{proposition} \label{prop1-section}
Let $\phi$ be an integrable Beltrami differential of $X$ with
supremum operator norm $\|\phi\|<1$. Given a harmonic
form $\sigma_0\in \mathbb H^{p,q}(X)$, if the smooth $(p,q)$-form $\sigma$ on $X$ is a solution of the
equation
\begin{align} \label{integralequation-section}
\sigma=\sigma_0-\bar{\partial}^*G\partial \left(\phi\lrcorner
\sigma\right)= \sigma_0-Ti_{\phi} \sigma,
\end{align}
then $\sigma$ is the solution of the equations
\eqref{obstructionequation2}.
\end{proposition}
\begin{proof}
We assume that $\sigma$ satisfies the equation $
\sigma=\sigma_0-\bar{\partial}^*G\partial \left(\phi\lrcorner
\sigma\right)$. 
Since $\sigma_0$ is harmonic, we have that $$\partial \sigma = \partial \sigma_0-\partial \bar{\partial}^*G\partial \left(\phi\lrcorner
\sigma\right)=\bar{\partial}^*G\partial^2 \left(\phi\lrcorner
\sigma\right)=0.$$
This proves the first equation of \eqref{obstructionequation2}.
Hence it remains to show that
\begin{align}\label{obstructionequation2-2}
\bar{\partial}\sigma=-\partial(\phi\lrcorner \sigma).
\end{align}
In fact, from the formulas in reviewed in Section
\ref{section-Hodge}, it follows that
\begin{align*}
\bar{\partial}\sigma&=-\bar{\partial}\bar{\partial}^*G\partial
\left(\phi\lrcorner \sigma\right)\\\nonumber
&=(\bar{\partial}^*\bar{\partial}-\triangle_{\bar{\partial}})G\partial(\phi\lrcorner
\sigma)\\\nonumber
&=(\bar{\partial}^*\bar{\partial}G-I+\mathbb{H})\partial(\phi\lrcorner
\sigma)\\\nonumber &=-\partial(\phi\lrcorner
\sigma)+\bar{\partial}^*\bar{\partial}G\partial(\phi\lrcorner
\sigma).
\end{align*}
Let
\begin{align*}
\Phi=\bar{\partial}\sigma+\partial(\phi\lrcorner \sigma).
\end{align*}
Then we have
\begin{align*}
\Phi&=\bar{\partial}\sigma+\partial(\phi\lrcorner
\sigma)\\\nonumber
&=\bar{\partial}^*\bar{\partial}G\partial(\phi\lrcorner
\sigma)\\\nonumber
&=-\bar{\partial}^*G\partial\bar{\partial}(\phi\lrcorner
\sigma)\\\nonumber
&=-\bar{\partial}^*G\partial((\bar{\partial}\phi)\lrcorner
\sigma+\phi\lrcorner \bar{\partial}\sigma)\\\nonumber
&=-\bar{\partial}^*G\partial\left(\frac{1}{2}[\phi,\phi]\lrcorner
\sigma+\phi\lrcorner \left(\Phi-\partial(\phi\lrcorner
\sigma)\right)\right)\\\nonumber
&=-\bar{\partial}^*G\partial\left(\phi\lrcorner \Phi\right),
\end{align*}
where in the last equality, we have used
(\ref{genralizedcartan_special}), $\partial \sigma=0$ and
$\partial^2=0$.

By quasi-isometry (\ref{quasi-iso}) and the condition
$\|\phi\|<1$, we have
\begin{align}
\|\Phi\|^2\leq \|\phi\lrcorner \Phi\|^2\leq
\|\phi\|\|\Phi\|^2< \|\Phi\|^2.
\end{align}
Then we get the contradiction $\|\Phi\|^2<\|\Phi\|^2$ unless
$\Phi=0$. Hence \eqref{obstructionequation2-2} is proved.
\end{proof}

Let $(X,\omega)$ be an $n$-dimensional compact
K\"ahler manifold with K\"ahler metric $\omega$,  and
$\Vert\cdot\Vert$ be the $L^{2}$-norm on smooth differential forms
$A^{p,q}(X)$ induced by the  metric $\omega$.
 Denote by $L^{p,q}_2(X)$ and $L^n_2(X)$ the $L^2$-completion of
$A^{p,q}(X)$ and $A^n(X)$ respectively. Then we have the decomposition
$$L^n_2(X) = \bigoplus_{p+q=n}L^{p,q}_2(X)$$
with Hodge filtration $F^p$ on $L^n_2(X)$ given by
$$F^pL^n_2(X) = \bigoplus_{p'+q'=n,\, p'\ge p}L^{p',q'}_2(X).$$

The quasi-isometry \eqref{quasi-iso} implies that $T=\bar{\partial}^*G\partial$ is an operator of
norm less than or equal to 1 on the Hilbert spaces of $L^2$ forms. So
we obtain
\begin{corollary} \label{Tinvertible}
Given a compact K\"ahler manifold $(X,\omega)$, let $\phi\in
A^{0,1}(X,\mathrm T^{1,0}X)$ be a Beltrami differential acting on the
Hilbert spaces of $L^2$ forms by contraction such that its
supremum operator norm $\|\phi\|<1$. Then the operator
$I+Ti_{\phi}$ is invertible on the Hilbert space of $L^2$ forms.
\end{corollary}

By Corollary \ref{Tinvertible}, equation
(\ref{integralequation-section}) has a unique solution in $L^{p,q}_2(X)$ for the given harmonic
$(p,q)$-form $\sigma_0$ on $X$, which is given by
\begin{align}\label{solution1}
\sigma=(I+Ti_{\phi})^{-1}\sigma_0=\sigma_0-Ti_{\phi}\sigma_0+\cdots +(-1)^k(Ti_{\phi})^k\sigma_0+\cdots.
\end{align}
Then, by extending Corollary~\ref{coro-obstructionequation2} and Proposition~\ref{prop1-section} to the $L^2$-completion of the space $A^{p+q}(X)$, we obtain that
$$
e^{i_{\phi}}(\sigma) \in F^p L_2^{p+q}(X_\phi)
$$
is $d$-closed.

In conclusion, we have the following theorem.
\begin{theorem}\label{solution}
Given any integrable Beltrami differential $\phi$ with the supremum operator norm
$\|\phi\|<1$, and any harmonic $(p,q)$-form $\sigma_0$ on the compact K\"ahler manifold $X$, we have that
\begin{align*}
\sigma(\phi)=e^{i_{\phi}}\left((I+Ti_{\phi})^{-1}\sigma_0\right) \in F^pL^{p+q}_2(X_\phi)
\end{align*}
is $d$-closed with $\sigma(0)=\sigma_0$.
In particular,
\begin{align*}
[\sigma(\phi)]&=[\mathbb H\sigma(\phi)]\\
&=[\sigma_0] + \left[\mathbb H\left(i_{\phi} (I+Ti_{\phi})^{-1}\sigma_0 \right)\right]+\cdots+\frac{1}{k!}\left[\mathbb H\left(i_{\phi}^k (I+Ti_{\phi})^{-1}\sigma_0 \right)\right]+\cdots
\end{align*}
as cohomological classes $[\cdot]$ in $F^pH^{p+q}(X_\phi,\C)$.
\end{theorem}
\begin{proof}
We only need to prove that 
$$\left[\mathbb H\left((I+Ti_{\phi})^{-1}\sigma_0 \right)\right]=[\sigma_0].$$
This follows from \eqref{solution1} and that 
$$\mathbb H (Ti_{\phi})^k= \left(\mathbb H \bar{\partial}^*G\partial \phi\right) (Ti_{\phi})^{k-1}=0.$$
Note that the representative $\mathbb H\sigma(\phi)$ of $[\sigma(\phi)]$ is a smooth form on $X_\phi$ due to the regularity of the elliptic operator $2\triangle_{\bar\partial}=\triangle_{d}$.
\end{proof}

\subsection{Closed formula for global section of Hodge bundles}\label{gsec def}
\noindent

In this subsection, we apply our closed formulas in Section \ref{subsection-extension} to get the sections of the Hodge bundles on $\tilde S$, on which we have an analytic family $\tilde f:\, \tilde\X \to \tilde S$ of polarized manifolds and a line bundle $\tilde {\mathcal L}$ on $\tilde{\X}$, such that $L_{t}= \tilde{\mathcal L}|_{X_{t}}$ for any fiber $(X_{t},L_{t})$ of $t\in \tilde S$.

Let $t_o=(X_{t_o}, L_{t_o})$ be a base point in ${\tilde{S}}$. Let $M$ be the background differential manifold of $X_{t_o}$. Let $\omega_{t_o}=\Theta(L_{t_o})$ be the corresponding K\"ahler form on $X_{t_o}$ induced by the curvature form of $L_{t_o}$, which can be considered as a symplectic form on $M$. 

For any $t=(X_{t}, L_{t})\in {\tilde{S}}$, we have a diffeomorphism
$$d_{t}:\, M\to X_{t}$$
which is not unique. Let $\omega_{t} = \Theta(L_{t})$ be the K\"ahler form on $X_{t}$. Then $d_{t}$ induces a symplectic form $d_{t}^*\omega_{t}$ on $M$.

Note that
$$[d_{t_o}^*\omega_{t_o}]=c_1(\mathcal L_T)|_{X_{t_o}} =c_1(\mathcal L_T)|_{X_{t}} =[d_{t}^*\omega_{t}]\in H^2(M).$$

By Theorem 2 of Moser in \cite{Moser}, we know that there exists a
continuous family of diffeomorphisms $f_t$ of $M$ with $t\in [0,1]$
and $f_0=\mathrm{id}$, such that $\omega_t= f_t^* d_{t_o}^*\omega_{t_o}$ and
$\omega_1 =d_{t}^*\omega_{t}$. Therefore by replacing the diffeomorphism $d_{t}$ with $d_{t}\circ f_1^{-1}$, we may assume that $d_{t}^*\omega_{t}=d_{t_o}^*\omega_{t_o}$ as symplectic forms on $M$.

\begin{lemma}\label{Lemma-moduli} Let $t_o=(X_{t_o}, L_{t_o})$ be a base point in ${\tilde{S}}$ and $t=(X_{t}, L_{t})\in {\tilde{S}}$ be any point. 
Choose a diffeomorphism $d_{t}:\, M\to X_{t}$, where $M$ is the background differential manifold of $X_{t_o}$, such that $d_{t}^*\omega_{t}=d_{t_o}^*\omega_{t_o}$ as symplectic forms on $M$.
Then the integrable almost complex structure on
$X_{t}$ is of finite distance from $X_{t_o}$. More precisely, there exists a Beltrami differential $\phi(t)\in A^{0,1}\left ( X_{t_o},\mathrm{T}^{1,0}X_{t_o}\right )$ with the supremum operator norm $\|{\phi(t)}\|<1$ as defined in \eqref{supremum operator norm}, such that $X_{t}=(X_{t_o})_{\phi(t)}$.
\end{lemma}
\begin{proof}
We identify
$$\mathrm{T}^{1,0}X_{t_o}\oplus \mathrm{T}^{0,1}X_{t_o}=\mathrm{T}_{\mathbb{C}}M$$
and let 
\begin{align*} \iota_1:\, \mathrm{T}_{\mathbb{C}}M\rightarrow
\mathrm{T}^{1,0}X_{t_o}, \  \iota_2:\,\mathrm T_{\mathbb{C}}M\rightarrow \mathrm{T}^{0,1}X_{t_o},
\end{align*}
be the corresponding projection maps.
 
The diffeomorphism $d_{t}:\, M\to X_{t}$ gives the isomorphism 
$$\mathrm{T}^{1,0}X_{t}\oplus \mathrm{T}^{0,1}X_{t}=\mathrm{T}_{\mathbb{C}}M,$$
under which we have that 
$$\omega_t(v,w)=\omega_{t_o}(v,w),\,\forall\, v,w\in \mathrm{T}^{1,0}X_{t}\oplus \mathrm{T}^{0,1}X_{t}.$$
Let $$\iota_1|_{\mathrm{T}^{1,0}X_{t}}:\, \mathrm{T}^{1,0}X_{t} \to \mathrm{T}^{1,0}X_{t_o},\,\iota_2|_{\mathrm{T}^{1,0}X_{t}}:\, \mathrm{T}^{1,0}X_{t} \to \mathrm{T}^{0,1}X_{t_o},$$
be the restrictions of the projection maps.

We let $(X,\omega)$ be any K\"ahler manifold, let $g$ be the
corresponding K\"ahler metric, let $J$ be the complex structure and
let $x\in X$ be a point. Then for any vector $v\in \mathrm{T}_x^{1,0}X$ such
that $v\ne 0$, we know that
\begin{equation*}
0<\Vert v\Vert^2=g(v,\bar v)=\omega(v,J(\bar
v))=-\sqrt{-1}\omega(v,\bar v).
\end{equation*}
Namely, $-\sqrt{-1}\omega(v,\bar v)>0$. Since $\omega$ is
skew-symmetric, we have that
$$-\sqrt{-1}\omega(u,\bar u)<0$$ for any
nonzero vector $u\in \mathrm{T}_x^{0,1}X$.

Now we pick any $v\in \mathrm{T}_x^{1,0}X_{t}$ such that $v\ne 0$. By the above
argument we have
\begin{align} \label{inequv}
-\sqrt{-1}\omega_{t_o}(v,\bar v)=-\sqrt{-1}\omega_t(v,\bar v)=\omega_{t}(v,J_t(\bar v))>0,
\end{align}
where $J_t$ is the complex structure on $X_t$.

Let $v_1=\iota_1(v)\in \mathrm{T}_x^{1,0}X_{t_o}$ and $v_2=\iota_2(v)\in
\mathrm{T}_x^{0,1}X_{t_o}$. It follows from type considerations that
$$\omega_{t_o}(v_1,\bar v_2)=0=\omega_{t_o}(v_2,\bar v_1).$$ If $v_1=0$ then
$$-\sqrt{-1}\omega_{t_o}(v,\bar v)= -\sqrt{-1}\omega_{t_o}(v_2,\bar v_2)<0$$ which
is
a contradiction. Thus we know that $v_1\ne 0$ which implies that $\iota_1\mid_{\mathrm{T}_x^{1,0}X_{t}}$ is a linear isomorphism. Thus,  the
integrable almost complex structure $X_{t}$ is of finite distance from
$X_{t_o}$ in the sense of Definition \ref{finite dist}. 

By the discussion in Section \ref{subsection-Beltrami}, the
integrable almost complex structure $X_{t}$ gives an integrable
Beltrami differential ${\phi(t)} \in A^{1,0}(X_{t_o},\mathrm{T}^{1,0}X_{t_o})$
determined by the map $$\bar{{\phi(t)}}(w)=-\iota_2\circ
\left(\iota_1|_{\mathrm{T}_x^{1,0}X_{t}}\right)^{-1}(w)$$ for $w\in
\mathrm{T}_x^{1,0}X_{t_o}$, such that $X_{t}=(X_{t_o})_{\phi(t)}$. Here by definition,
$$\bar{\phi(t)}(v)= \overline{\phi(t)(\bar{v})}.$$

Let $w\in \mathrm{T}_x^{1,0}X_{t_o}$ and $v=\left(\iota_1|_{\mathrm{T}_x^{1,0}X_{t}}\right)^{-1}(w)\in \mathrm{T}_x^{1,0}X_{t}$. Then $v=w+\bar{{\phi(t)}}(w)$ and the inequality in (\ref{inequv}) implies
\begin{equation}\label{norm}
0<\sqrt{-1}\omega_{t_o}(\bar{{\phi(t)}}(w),\bar{\bar{{\phi(t)}}(w)})<-\sqrt{-1}\omega_{t_o}(w,\bar w).
\end{equation}

Let the $(\cdot,\cdot)_{x}$, $x\in X$, be the Hermitian form on $A^{p,q}(X)$, which is induced by the K\"ahler metric $\omega_{t_o}$. Then the inequality (\ref{norm}) implies that 
$$(i_{\phi(t)}\alpha,i_{\phi(t)}\alpha)_{x}<(\alpha,\alpha)_{x}, \,\forall\, \alpha \in A^{p,q}(X),$$
that is, the pointwise operator norm $\|\phi(t)\|_{2,x}<1$ on $A^{p,q}(X)$.
Since $X$ is compact, the supremum operator norm as defined in \eqref{supremum operator norm} satisfies that
$$\|{\phi(t)}\|=\sup_{x\in X}\|\phi(t)\|_{2,x}=\|\phi(t)\|_{2,x_{0}}<1,$$
for some $x_{0}\in X$,
and hence 
$$\|i_{\phi(t)}\alpha\|\le \|{\phi(t)}\|\|\alpha\| <\|\alpha\|,\, \forall\, \alpha\in A^{p,q}(X),$$
where $\|\cdot\|$ is the $L^{2}$-norm on $A^{p,q}(X)$ induced by the K\"ahler metric $\omega_{t_o}$.
This concludes the proof.
\end{proof}

For a polarized manifold $(X,L)$, we are interested in the primitive cohomology group $$H_{\mathrm{pr}}^{n}(X,\C)=\{\alpha \in H^{n}(X,\C):\, c_{1}(L)^{d-n+1}\wedge\alpha=0\}, \, n\le d=\dim_{\C} X,$$
which is isomorphic to the subspace
\begin{equation}\label{hpp}
\mathbb H_{\mathrm{pr}}^{n}(X)=\{\alpha\in \mathbb H_{\mathrm{pr}}^{n}(X):\, \omega^{d-n+1}\wedge \alpha=0\}
\end{equation}
of the space of harmonic $n$-forms, where $\omega$ is the curvature form of $L$.

Let $$\mathbb H_{\mathrm{pr}}:\, A^{n}(X)\to \mathbb H_{\mathrm{pr}}^{n}(X)$$
be the corresponding projection map.

\begin{theorem} \label{Theorem-closedsection}
Given a base point $t_o=(X_{t_o}, L_{t_o})$ in ${\tilde{S}}$ and a harmonic and primitive $(p,n-p)$-form $\sigma_0$ on $X_{t_o}$, there
is a canonical section $\Sigma$ of the Hodge bundle $\mathcal{F}^{p}$ on ${\tilde{S}}$,
such that $\Sigma(t_o)=[\sigma_0]$ and that for any point $t=(X_{t},L_{t})\in {\tilde{S}}$, $\Sigma(t)$ is represented by the de Rham cohomology class in $F^{p}H_{\mathrm{pr}}^{n}(X_{t},\C)$,
\begin{eqnarray}
[\mathbb H_{\mathrm{pr}}\sigma({\phi(t)})]&=&\left[\mathbb H_{\mathrm{pr}} \left(e^{i_{\phi(t)}}\left((I+Ti_{\phi(t)})^{-1}\sigma_0\right)\right)\right]\nonumber \\
&=&[\sigma_0] + \left[\mathbb H_{\mathrm{pr}}\left(i_{\phi(t)} (I+Ti_{\phi(t)})^{-1}\sigma_0 \right)\right]+\cdots \label{Theorem-closedsection 1}
\\
&&+\frac{1}{k!}\left[\mathbb H_{\mathrm{pr}}\left(i_{\phi(t)}^k (I+Ti_{\phi(t)})^{-1}\sigma_0 \right)\right]+\cdots \nonumber
\end{eqnarray}
where ${\phi(t)}$ is the Beltrami differential associated to $X_{t}$, and
$T=\bar{\partial}^*G\partial $ is the operator from the Harmonic theory
on $X_{t_o}$ with K\"ahler metric $\omega_{t_o}$. 
\end{theorem}
\begin{proof}
Choose a diffeomorphism $d_{t}:\, M\to X_{t}$ as in Lemma \ref{Lemma-moduli}.
Then by Lemma \ref{Lemma-moduli}, we know that associated to the
complex structure on $X_{t}$, there exists the Beltrami differential
${\phi(t)}\in A^{0,1}(X_{t_o},\mathrm{T}^{1,0}X_{t_o})$, with $\|{\phi(t)}\|<1$, where the
norm is taken with respect to the metric $\omega_{t_o}$ on $X_{t_o}$.

From Corollary \ref{solution}, since $\|{\phi(t)}\|<1$ on $X_{t_o}$, we
deduce that given the harmonic $(p,n-p)$-form $\sigma_0$, there exist extensions of $\sigma_0$ as $d$-closed $n$-form 
$$\sigma({\phi(t)})=e^{i_{\phi(t)}}\left((I+Ti_{\phi(t)})^{-1}\sigma_0\right).$$

The simply-connectedness of ${\tilde{S}}$ and Ehresmann's theorem imply that de Rham cohomology class $[\sigma({\phi(t)})]\in d_{t_o}^*( F^{p}H^{n}(X_{t},\C))$ is independent of the choice of the diffeomorphism $d_{t}:\, M\to X_{t}$, which we simply denote by $[\sigma({\phi(t)})]\in  F^{p}H^{n}(X_{t},\C)$.

Note that the harmonic projection map \eqref{hpp} onto the subspace of harmonic and primitive forms is compatible with the Hodge filtrations,
$$\mathrm H_{\mathrm{pr}}:\, F^{p}A^{n}(X_{t})\to F^{p}\mathbb H_{\mathrm{pr}}^{n}(X_{t}).$$
Hence $$[\mathbb H_{\mathrm{pr}}\sigma({\phi(t)})]\in F^{p}H_{\mathrm{pr}}^{n}(X_{t},\C).$$

Therefore the formula \eqref{Theorem-closedsection 1} holds on ${\tilde{S}}$.
\end{proof}

%
%
%
%


\section{A conjecture of Griffiths}\label{main proof}
In this section, we establish a partial result towards Griffiths’ Conjecture 10.1 in \cite{Griffiths4}, using the sections of the Hodge bundles constructed via the period map and deformation theory.

Now we take the blocks of the adapted basis $\eta$ of the Hodge decomposition at the base point $t_o=(X_{t_o},L_{t_o})$ in $\tilde S$ as
$$\eta=\{\eta_{(0)}^T,\eta_{(1)}^T,\cdots,\eta_{(n)}^T\}^T,$$
where 
$$\eta_{(p)}=[\tilde\eta_{(p)}]=\{[\tilde\eta_{f^{-p +n+1}}], \cdots, [\tilde\eta_{f^{-p +n}-1}]\}^T$$ 
is represented by harmonic $(n-p,p)$-forms 
$$\tilde\eta_{i}\in \mathbb H_{\mathrm{pr}}^{n-p,p}(X_{t_o}),\,f^{-p +n+1} \le i\le f^{-p +n}-1,$$
for $0\le p \le n$.

For any point $t=(X_{t},L_{t})\in {\tilde{S}}$, Theorem \ref{Theorem-closedsection} implies that there exist sections $\tO_{(p)}$ of Hodge bundles $\mathcal{F}^{n-p}$ on ${\tilde{S}}$, $0\le p \le n$, given by
\begin{align}
\tO_{(p)}(t)=&\left[\mathbb H_{\mathrm{pr}} \left(e^{i_{\phi(t)}}\left((I+Ti_{\phi(t)})^{-1}\tilde\eta_{(p)}\right)\right)\right]\nonumber \\
=&\eta_{(p)} + \left[\mathbb H_{\mathrm{pr}}\left(i_{\phi(t)} (I+Ti_{\phi(t)})^{-1}\tilde\eta_{(p)} \right)\right]+\cdots \label{def section} \\
&\, \, +\frac{1}{k!}\left[\mathbb H_{\mathrm{pr}}\left(i_{\phi(t)}^k (I+Ti_{\phi(t)})^{-1}\tilde\eta_{(p)} \right)\right]+\cdots \nonumber 
\end{align}
with initial data $\tO_{(p)}(t_0)=\eta_{(p)}$,
where $\phi(t)$ is the Beltrami differential associated to $X_{t}$. 

Note that as cohomological classes at the base point $X_{t_o}$, the summands of $\tO_{(\alpha)}(t)$ satisfy that
\begin{eqnarray*}
  \left[\mathbb H_{\mathrm{pr}}\left(i_{\phi(t)}^k (I+Ti_{\phi}(t))^{-1}\tilde\eta_{(p)} \right)\right]&\in & H_{\mathrm{pr}}^{n-p-k,p+k}(X_{t_o})
\end{eqnarray*}
for $k\ge 1$.

\begin{lemma}\label{def basis}
  The sections $\tO_{(0)}(t),\cdots,\tO_{(p)}(t)$ given in \eqref{def section} is a basis of $F^{n-p} H_{\mathrm{pr}}^n(X_{t},\C)$.
\end{lemma}
\begin{proof}
  By considering the dimensions, we only need to show that $\tO_{(0)}(t),\cdots,\tO_{(p)}(t)$ are linearly independent. 
  
Let $\alpha_{(0)}(t), \cdots, \alpha_{(p)}(t)$ be the row vectors such that
$$\alpha_{(0)}(t)\tO_{(0)}(t)+ \cdots + \alpha_{(p)}(t)\tO_{(p)}(t)=0.$$
Then by comparing types of the following equation in $H_{\mathrm{pr}}^{n,0}(X_{t_o}),\cdots, H_{\mathrm{pr}}^{n-p,p}(X_{t_o})$, 
\begin{eqnarray*}
0&=&\sum_{0\le i\le p}\alpha_{(i)}\left\{\eta_{(i)}+\sum_{1\le k\le n-i}\frac{1}{k!}\left[\mathbb H_{\mathrm{pr}}\left(i_{\phi(t)}^k (I+Ti_{\phi(t)})^{-1}\tilde\eta_{(i)} \right)\right]\right\}\\
&=&\sum_{0\le i\le p}\left\{\alpha_{(i)}\eta_{(i)}+\sum_{1\le k\le i}\frac{1}{k!}\alpha_{(i-k)}\left[\mathbb H_{\mathrm{pr}}\left(i_{\phi(t)}^k (I+Ti_{\phi(t)})^{-1}\tilde\eta_{(i-k)} \right)\right]\right\}\\
&&+\sum_{0\le i\le p}\alpha_{(i)}\sum_{p-i< k\le n-i}\frac{1}{k!}\left[\mathbb H_{\mathrm{pr}}\left(i_{\phi(t)}^k (I+Ti_{\phi(t)})^{-1}\tilde\eta_{(i)} \right)\right],
\end{eqnarray*}
we have that
\begin{eqnarray*}
\left\{ \begin{aligned}& \alpha_{(0)}(t)\eta_{(0)} =0\\
&\alpha_{(1)}\eta_{(1)}+\alpha_{(0)}\left[\mathbb H_{\mathrm{pr}}\left(i_{\phi(t)} (I+Ti_{\phi}(t))^{-1}\tilde\eta_{(0)} \right)\right] = 0\\
&\dots \dots\\
&\alpha_{(p)}\eta_{(p)}+\sum_{1\le k\le p}\frac{1}{k!}\alpha_{(p-k)}\left[\mathbb H_{\mathrm{pr}}\left(i_{\phi(t)}^k (I+Ti_{\phi(t)})^{-1}\tilde\eta_{(p-k)} \right)\right]=0.
\end{aligned} \right.
\end{eqnarray*}
Hence, by induction, we conclude that $\alpha_{(0)}(t)=0, \cdots, \alpha_{(p)}(t)=0$, which completes the proof of the lemma.
\end{proof}

When $t\in {\tilde{S}}^{\vee}$, we have constructed the sections 
\begin{equation}\label{lm section'}
  \Omega_{(p)}(t)=\eta_{(p)}+\sum_{k\ge 1} \Phi^{(p,p+k)}(t)\cdot \eta_{(\beta)},\,0\le p \le n
\end{equation}
of Hodge bundles $\mathcal F^{n-p}$ in \eqref{lm section} using the matrix representation of the image $\P(t)$ in $N_-$, $$\P(t)=\left(\Phi^{(p,q)}(t)\right)_{0\le p, q\le n}\in N_-\cap D.$$

\begin{lemma}\label{def=hodge}
The sections $\Omega_{(p)}$ and $\tO_{(p)}$ of the Hodge bundle $\mathcal F^{n-p}$ coincide on ${\tilde{S}}^{\vee}$ for all $0\le p \le n$.
\end{lemma}
\begin{proof}
  Let $t\in {\tilde{S}}^{\vee}$ be any point.
  From Lemma \ref{def basis}, there exist matrices $A_{(0)}(t), \cdots, A_{(p)}(t)$ such that
\begin{equation}\label{O=tilde O}
\Omega_{(p)}(t)= A_{(0)}(t)\tO_{(0)}(t)+ \cdots +A_{(p)}(t)\tO_{(p)}(t).
\end{equation}
Then the left-hand side equals
$$\eta_{(p)}+\sum_{k\ge 1} \Phi^{(p,p+k)}(t)\cdot \eta_{(\beta)},$$
while the right-hand side equals
\begin{eqnarray*}
&&\sum_{0\le i\le p}A_{(i)}\left\{\eta_{(i)}+\sum_{1\le k\le n-i}\frac{1}{k!}\left[\mathbb H_{\mathrm{pr}}\left(i_{\phi(t)}^k (I+Ti_{\phi(t)})^{-1}\tilde\eta_{(i)} \right)\right]\right\}\\
&=&\sum_{0\le i\le p}\left\{A_{(i)}\eta_{(i)}+\sum_{1\le k\le i}\frac{1}{k!}A_{(i-k)}\left[\mathbb H_{\mathrm{pr}}\left(i_{\phi(t)}^k (I+Ti_{\phi(t)})^{-1}\tilde\eta_{(i-k)} \right)\right]\right\}\\
&&+\sum_{0\le i\le p}A_{(i)}\sum_{p-i< k\le n-i}\frac{1}{k!}\left[\mathbb H_{\mathrm{pr}}\left(i_{\phi(t)}^k (I+Ti_{\phi(t)})^{-1}\tilde\eta_{(i)} \right)\right].
\end{eqnarray*}
Hence by comparing types of \eqref{O=tilde O} in $H^{n,0}(X_{t_o}),\cdots, H^{n-p,p}(X_{t_o})$, we have that 
$$\eta_{(p)}=\sum_{0\le i\le p}\left\{A_{(i)}\eta_{(i)}+\sum_{1\le k\le i}\frac{1}{k!}A_{(i-k)}\left[\mathbb H_{\mathrm{pr}}\left(i_{\phi(t)}^k (I+Ti_{\phi(t)})^{-1}\tilde\eta_{(i-k)} \right)\right]\right\},$$
which implies that
 $$ A_{(0)}(t)=O,\cdots, A_{(p-1)}(t)=O, A_{(p)}(t)=I.$$
 Therefore $\Omega_{(p)}(t)=\tO_{(p)}(t)$ for any $t\in \tilde S^{\vee}$ and $0\le p\le n$.
\end{proof}

Now we can prove our first main result.

\begin{theorem}\label{main}
Let $f:\, \mathcal X \to S$ be an analytic family of polarized manifolds over a complex manifold $S$ and $\pi:\,\tilde S\to S$ be the universal cover. Then the image of the lifted period map 
$$\Phi:\,{\tilde{S}}\to D$$
lies in the complex Euclidean space $N_-$.
\end{theorem}
\begin{proof}
Let $t\in {\tilde{S}}$ be any point.
From Proposition \ref{codimension}, ${\tilde{S}}^{\vee}\subseteq {\tilde{S}}$ is open dense in ${\tilde{S}}$ with ${\tilde{S}}\setminus {\tilde{S}}^{\vee}$ as an analytic subset of ${\tilde{S}}$, we can find a curve $\gamma:\, [0, 1]\to {\tilde{S}}$ such that $\gamma(0)=t_o$, $\gamma(1)=t$ and $\gamma([0,1))\subset {\tilde{S}}^{\vee}$. 

By the definition of ${\tilde{S}}^{\vee}$, we have that $\P(\gamma(s))\in N_-\cap D$ for any $s\in [0,1)$. We denote
$$\P(\gamma(s))=(\Phi^{(p,q)}(s))_{0\le p,q\le n}\in N_{-},\, s\in [0,1)$$
In the following blocks of $\P(\gamma(s))$ as $s\to 1$,  
\begin{align*}
\P(\gamma(s))=\left(\begin{array}[c]{ccccccc}
I & \Phi^{(0,1)}(s) & \Phi^{(0,2)}(s) & \cdots& \Phi^{(0,n-1)}(s) & \Phi^{(0,n)}(s)\\ 
0& I& \Phi^{(1,2)}(s)&\cdots &\Phi^{(1,n-1)}(s)& \Phi^{(1,n)}(s)\\ 
0& 0&I &\cdots &\Phi^{(2,n-1)}(s)&\Phi^{(2,n)}(s) \\ 
\vdots &\vdots & \vdots& \ddots& \vdots&\vdots \\ 
0&0 &0 &\cdots & I & \Phi^{(n-1,n)}(s)\\
 0&0 &0 &\cdots &  0&I
\end{array}\right),
\end{align*}
we only need to show that the limits
\begin{equation}\label{blocks finite}
\lim_{s\to 1}\Phi^{(p,p+k)}(s)<\infty \text{ for } 0\le p<p+k \le n .
\end{equation}

For fixed $p$, the blocks $\Phi^{(p,p+k)}(s)$ define the section of the Hodge bundle 
$$\Omega_{(p)}(s)=\eta_{(p)}+\sum_{k\ge 1} \Phi^{(p,p+k)}(s)\cdot \eta_{(\beta)},\,s\in [0,1)$$
which is identified with $\tO_{(p)}(s)$ by Lemma \ref{def=hodge}. 

From Theorem \ref{Theorem-closedsection} we have that the section $\tO_{(p)}(s)$ of the Hodge bundle exists globally on $\tilde S$ and can be continued to $s=1$, which implies \eqref{blocks finite}.
\end{proof}

\begin{remark}
The assumption that we are dealing with a family of polarized manifolds cannot be removed in Theorem \ref{main}. Indeed, a projective manifold $X$ can be deformed to its complex conjugate $\bar{X}$ in an analytic family $f:\,\mathcal X\to S$ of projective manifolds over a connected base $S$, where $\Tan \bar{X}=\mathrm{T}^{0,1}X$.
Let $\tilde S$ be the universal cover of $S$. Then there exists a period map $\Phi:\, \tilde S\to D$, where $D$ is the period domain of (non-polarized) Hodge structures.
If we take $[X]$ as the base point in $\tilde S$, then the image $\Phi([\bar X])$ does not lie in $N_{-}$.

For example, the Teichm\"uller space $\mathcal T_{g}$ of Riemann surfaces, which is the universal cover of the moduli space $\mathcal M_{g}$, is connected and contains points $X\neq \bar X$. Another interesting example is given by the hyperk\"ahler manifold $(X;I,J,K)$, where $I,J,K$ are the three complex structures on $X$. In this case, one can construct a family of complex structures $\mathcal X\to \mathbb P^{1}$ such that for any $z=[x_{1},x_{2},x_{3}]\in \mathbb P^{1}$, where $(x_{1},x_{2},x_{3})$ are the coordinates of $\mathbb R^{3}$, the complex structure on $X_{z}$ is given by
$x_{1}I+x_{2}J+x_{3}K$. Then $X_{[1,0,0]}=(X;I)$ is deformed to its complex conjugate $X_{[-1,0,0]}=(X;-I)$ along $\mathbb P^{1}$.
\end{remark}

\begin{theorem}\label{main2}
Let the assumptions be as in Theorem~\ref{main}. Then the sections
\begin{equation}\label{lm section''}
  \Omega_{(p)}(t)=\eta_{(p)}+\sum_{k\ge 1} \Phi^{(p,p+k)}(t)\cdot \eta_{(\beta)},\quad 0\le p \le n
\end{equation}
of the Hodge bundles $\mathcal F^{n-p}$ over $t\in \tilde S$, defined via the matrix representation of the image $\P(t)$ in $N_-$, 
are identified with the sections
\begin{equation}\label{def section'}
  \tO_{(p)}(t)=\left[\mathbb H_{\mathrm{pr}} \left(e^{i_{\phi(t)}}\left((I+Ti_{\phi}(t))^{-1}\tilde\eta_{(p)}\right)\right)\right],
\end{equation}
constructed using the Beltrami differential $\phi(t)$ associated to $X_{t} = (X_{t_{o}})_{\phi(t)}$, where $\tilde\eta_{(p)}$ is a column vector of harmonic $(n-p,p)$-forms on $X_{t_o}$ representing $\eta_{(p)}$. In particular, the sections $\tO_{(p)}(t)$ are holomorphic on ${\tilde{S}}$.
\end{theorem}

\begin{remark}
 It is interesting to observe that the sections of Hodge bundles in \eqref{lm section''} constructed from period maps have a nice infinitesimal property, as in Lemma \ref{lm derivative lemma}, while the sections of Hodge bundles in \eqref{def section'} constructed from deformation theory have a nice global property.
\end{remark}


\section{Affine structures on the moduli spaces and examples}\label{affine and examples section}
In this section, we review the notions of moduli spaces, level-$m$ structures, and Teichmüller spaces for polarized manifolds, together with the relevant Hodge theory and period maps. Building on Theorem \ref{main4}, we construct a holomorphic map from the Teichm\"uller space of Calabi--Yau type manifolds to a complex Euclidean space, which induces a global complex affine structure.

\subsection{Moduli spaces and period maps}\label{moduli and period}
\

In this section, we introduce the notions of moduli spaces, level $m$ structure and Teichm\"uller spaces for polarized manifolds. We also review Hodge theory and define the period maps from the moduli spaces and Teichm\"uller spaces. 

The moduli space $\M_{\mathrm{p}}$ of polarized manifolds is the complex analytic space parameterizing the isomorphism class of polarized manifolds with the isomorphism defined by
$$(X,L)\sim (X',L')\EQ \exists \ \text{biholomorphic map} \ f : X\to X' \ \text{s.t.} \ f^*L'=L .$$

We fix a lattice $H_{\mathbb Z}$ with a pairing $Q_{0}$, where $H_{\mathbb Z}$ is isomorphic to $H_{\mathrm{pr}}^n(X_{0},\mathbb{Z})/\text{Tor}$ for some $X_{0}$ in $\M$ and $Q_{0}$ is defined by the cup-product.
For a polarized manifold $(X,L)\in \M_{\mathrm{p}}$, we define a marking $\gamma$ as an isometry of the lattices
\begin{equation}\label{marking}
\gamma :\, (H_{\mathbb Z}, Q_{0})\to (H_{\mathrm{pr}}^n(X,\mathbb{Z})/\text{Tor},Q).
\end{equation}

For any integer $m\geq 3$, we follow the definition of Szendr\"oi \cite{Szendroi}
 to define an $m$-equivalent relation of two markings on $(X,L)$ by
$$\gamma\sim_{m} \gamma' \text{ if and only if } \gamma'\circ \gamma^{-1}-\text{Id}\in m \cdot\text{End}(\H^n(X,\mathbb{Z})/\text{Tor}),$$
and denote $[\gamma]_{m}$ to be the set of all the $m$-equivalent classes of $\gamma$.
Then we call $[\gamma]_{m}$ a level $m$
structure on the polarized manifold $(X,L)$.

Two polarized manifolds with level $m$ structure $(X,L,[\gamma]_{m})$ and $(X',L',[\gamma']_{m})$ are said to be isomorphic, or equivalent, if there exists a biholomorphic map $f : X\to X'$ such that
\begin{equation}\label{level m equi}
  f^*L'=L \text{ and } f^*\gamma'\sim_{m} \gamma,
\end{equation}
where $f^*\gamma'$ is given by $$\gamma':\, (H_{\mathbb Z}, Q_{0})\to (\H^n(X',\mathbb{Z})/\text{Tor},Q)$$ composed with the induced map
$$f^*:\, (\H^n(X',\mathbb{Z})/\text{Tor},Q)\to (\H^n(X,\mathbb{Z})/\text{Tor},Q).$$
We denote by $[X, L, [\gamma]_{m}]$ the isomorphism class of the polarized manifolds with level $m$ structure $(X, L, [\gamma]_{m})$.

The moduli space of polarized manifolds with level $m$ structure is the analytic space which parameterizes the isomorphism class of polarized manifolds with level $m$ structure, where $m\ge 3$.
Let $\mathscr{L}_{m}$ be the moduli space of polarized manifolds with level $m$ structure, $m\ge 3$, which contains the given polarized manifold $(X,L)$.

\begin{definition}\label{T-class} A polarized manifold  $(X,L)$ is said to belong to the T-class, if the irreducible component $\Z$ of $\mathscr{L}_{m}$ containing $(X,L)$ is a
complex manifold, on which there is a universal analytic family $f_{m}:\,\U_{m}\to \Z$ with a line bundle $\mathcal L_m$ on $\U_m$ for all $m\ge m_{0}$, where $m_{0}\ge 3$ is some integer.
\end{definition}
Without loss of generality, we may simply take $m_{0}=3$.

In this paper we only consider the polarized manifold $(X,L)$ belonging to the T-class, and the smooth moduli space $\Z$ for $m\geq 3$. Our purpose of introducing this notion is for us to  work on the smooth covers of the moduli spaces.


Let $\T^m$ be the universal covering of $\Z$ with covering map $\pi_m:\, \T^m\to \Z$. Then we have an analytic family $g_m:\, \V^m \to \T^m$ such that the following diagram is cartesian
$$ \begin{CD}
\V^m @> >> \U_m \\
@V Vg_mV  @V Vf_mV \\
\T^m @> >> \Z
\end{CD} $$
i.e. $\V^m=\U_m \times_{\Z}\T^m$. The family $g_m$ is called the pull-back family.

The proof of the following lemma is obvious.
\begin{lemma}
The space $\T^m$ is a connected complex manifold on which there is the pull-back family $g_m:\, \V^m \to \T^m$ with a line bundle $\tilde{\mathcal L}_m$ on $\V^m$.
\end{lemma}

The following lemma proves that the space $\T^m$ is independent of the level $m$. 

\begin{lemma}\label{independent of m}
The space $\T^m$ does not depend on the choice of $m$. From now on, we simply denote $\T^m$ by $\T$, the analytic family by $g:\, \V \to \T$ and the covering map by $\pi_m:\, \T\to \Z$.
\end{lemma}
\begin{proof} 
The proof  uses the construction of moduli space with level $m$ structure, see Lecture 10 of \cite{Popp}, or pages 692 -- 693 of \cite{Szendroi}.

Let $m_1$ and $m_2$ be two different integers, and
$$f_{m_{1}}:\,\U_{m_1}\to \mathcal{Z}_{m_1}, \ f_{m_{2}}:\,\U_{m_2}\to \mathcal{Z}_{m_2}$$ be two analytic families with level $m_1$ structure and level $m_2$ structure respectively.
Let $\T^{m_{1}}$ and $\T^{m_{2}}$ be the universal covering space of $\mathcal{Z}_{m_1}$ and $\mathcal{Z}_{m_2}$ with the pull back family
$$g_{m_{1}}:\,\V^{m_1}\to \mathcal{\T}^{m_1}, \ g_{m_{2}}:\,\V^{m_2}\to \mathcal{\T}^{m_2}$$
of $f_{m_{1}}$ and $f_{m_{2}}$ respectively.
Let $m=m_1m_2$ and consider the analytic family $$f_{m}:\,\U_{m}\to \mathcal{Z}_{m}.$$
From the discussion in Page 130 of \cite{Popp} or 692 -- 693 of \cite{Szendroi}, we know that $\mathcal{Z}_{m}$ is a covering space of both $\mathcal{Z}_{m_1}$ and $\mathcal{Z}_{m_2}$, and the analytic family $f_{m}$ over $\Z$ is the pull-back family of both $f_{m_{1}}$ and $f_{m_{2}}$ via the corresponding covering maps.

Let $\T$ be the universal covering space of $\mathcal{Z}_{m}$ with the pull-back family
$$g:\,\V\to \mathcal{\T}$$ of $f_{m}$. Since $\mathcal{Z}_{m}$ is a covering space of both $\mathcal{Z}_{m_1}$ and $\mathcal{Z}_{m_2}$, we conclude that $\T$ is the universal cover of both $\mathcal{Z}_{m_1}$ and $\mathcal{Z}_{m_2}$, i.e. $$\T^{m_1}\simeq \T^{m_2} \simeq \mathcal{T},$$
and that the analytic family $g$ is identified with the analytic families $g_{m_{1}}$ and $g_{m_{2}}$ via the above isomorphisms.
\end{proof}

Due to the above lemma, we can denote $\T=\T^m$ for any $m\ge 3$, and call $\T$ the Teichm\"uller space of polarized manifolds.


For the family $f_m:\,  \U_m \to \Z$, we have the period map
$$\Phi_{\Z} :\, \Z \to \Gamma_{m} \backslash D$$
as given in Section \ref{Lie}, where 
$$\rho :\,  \pi_1(\Z)\to \Gamma_{m} \subseteq \text{Aut}(H_{\mathbb{Z}},Q),$$
is the monodromy representation.

We can lift the period map onto the universal cover $\T$ of $\Z$, to get the lifted period map $\Phi :\, \T \to D$ such that the diagram
$$\xymatrix{
\T \ar[r]^-{\Phi} \ar[d]^-{\pi_m} & D\ar[d]^-{\pi_{D}}\\
\Z \ar[r]^-{\Phi_{\Z}} & \Gamma_{m} \backslash D
}$$
is commutative.

From Theorem \ref{main}, we have the following:

\begin{theorem}\label{main4}
Let $\T$ be the Teichm\"uller space of polarized manifolds as defined in Section \ref{moduli and period}. Then the image of the period map 
$$\Phi:\, \T \to D$$
lies in the complex Euclidean space $N_-$.
\end{theorem}




\subsection{Affine structure on the Teichm\"uller space of Calabi--Yau type manifolds}\label{affine}
\

In this section, we first apply Theorem \ref{main4} in
the previous section to construct an affine structure on the Teichm\"uller space $\T$ of Calabi--Yau type manifolds.

\begin{definition}\label{defn of gcyt}
Let $X$ be a projective manifold with $\dim_{\mathbb{C}}X=d$. We call $X$ a Calabi--Yau type manifold if it satisfies the following conditions:
\begin{enumerate}
\item[(i)] There exists some integer $k$ with $[d/2]< k\leq d$ such that
$$H_{\mathrm{pr}}^{l,d-l}(X)=0$$
for $k< l \le d$, and $\dim_{\mathbb{C}}H_{\mathrm{pr}}^{k,d-k}(X)=1$;

\item[(ii)] For any generator $\Omega\in H_{\mathrm{pr}}^{k,d-k}(X)$, the contraction map
\begin{align*}
\lrcorner:\, H^{1}(X,\Theta_X)\to H_{\mathrm{pr}}^{k-1,d-k+1}(X), \quad\quad
\phi\mapsto \phi\lrcorner\Omega
\end{align*}
is an isomorphism.
\label{condition2}
\end{enumerate}
\end{definition}

A Calabi--Yau manifold is defined to be a projective manifold $X$ with trivial canonical bundle $\Omega_{X}^{\dim_{\C}X}$ such that
$H^{i}(X,\mathcal O_{X})=0$ for $0< i<\dim_{\C}X$.
Thus, a Calabi--Yau manifold is, of course, a Calabi--Yau type manifold with $k=d$ in the above definition. However, there exist further examples, such as certain hypersurfaces in the projective space $\mathbb P^{n}$, whose canonical bundle is nontrivial but which satisfy the conditions in the above definition.

Before presenting the main theorem of this section, we recall the definition of a complex affine structure on a complex manifold.
\begin{definition}
Let $M$ be a complex manifold of complex dimension $n$. If there
is a coordinate cover $\{(U_i,\,\phi_i);\, i\in I\}$ of M such
that $\phi_{ik}=\phi_i\circ\phi_k^{-1}$ is a holomorphic affine
transformation on $\mathbb{C}^n$ whenever $U_i\cap U_k$ is not
empty, then $\{(U_i,\,\phi_i);\, i\in I\}$ is called a complex
affine coordinate cover on $M$ and it defines a holomorphic affine structure on $M$.

In particular, if there exists a holomorphic map $\phi:\, M\to \C^n$ such that the tangent map $d\phi$ is non-degenerate everywhere, then the affine structure on $M$ induced by $\phi$ is called a global affine structure.
\end{definition}

The following main theorem of this section is an immediate consequence of Remark \ref{trans remark}.

\begin{theorem}\label{affine structure} 
Let $X$ be a Calabi--Yau type manifold.
Assume that the Teichm\"uller space $\T$ of Calabi--Yau type manifolds containing $X$ exists as in Section \ref{moduli and period}.
Then there exists a global complex affine structure on $\T$.
\end{theorem}
\begin{proof}
We will construct the affine structure on $\T$ in the proof. 

First, we take the Hodge structure on $H=H^d_{\mathrm{pr}}(X,\C)$ and tensor it with the Tate Hodge structure
$$T(d-k)=T^{d-k,d-k}\simeq \C,$$
to obtain a Hodge structure
$$H=\H^{n,0}(X)\oplus \H^{n-1,1}(X)\oplus \cdots \oplus \H^{0,n}(X)$$
of weight $n=2k-d$. Then we have that $H^{n,0}(X)$ is one-dimensional with a generator $\Omega$ such that 
\begin{align}\label{CY type iso}
(\cdot)\lrcorner\Omega:\, H^{1}(X,\Theta_X)\to \H^{n-1,1}(X),
\end{align}
is an isomorphism.

By Theorem \ref{main4}, we have a period map $$\Phi:\, \T \to N_{-}\cap D,\, t\mapsto \left(\Phi^{(p,q)}(t)\right)_{0\le p,q\le n}$$
be the period map from Teichm\"uller space $\T$ of Calabi--Yau type manifolds.

We define the holomorphic map 
$$\Psi:\, \T \to \C^{N}$$
by mapping $t\in \T$ to $\Phi^{(0,1)}(t)$, where $N=\dim_{\C} H^{1}(X,\Theta_X)$ and the block $\Phi^{(0,1)}(t)$ is a row vector of dimension $N$.

Let $(t_{\mu})_{1\le \mu \le N}$ be a holomorphic coordinate on $\T$.
From Remark \ref{trans remark}, we have that 
$$ (d\P)_{t}\left(\frac{\partial}{\partial t_{\mu}}\right)=\bigoplus_{0\le p\le n-1}\frac{\partial\Phi^{(p,p+1)}}{\partial t_\mu}(t),$$
where 
$$\frac{\partial\Phi^{(0,1)}}{\partial t_\mu}(t)\in \mathrm{Hom}(\H^{n,0}(X_{t}),\H^{n-1,1}(X_{t}))$$
is given by the contraction $\rho\left(\frac{\partial}{\partial t_{\mu}}\right)\lrcorner\Omega_{t}$, where $\Omega_{t}\in H_{\mathrm{pr}}^{n,0}(X_{t})$ is a generator and
$$\rho:\, \mathrm T_{t}\Delta \to H^{1}(X_{t},\Theta_{X_{t}})$$
is the Kodaira--Spencer map. 

Therefore from the local Torelli assumption \eqref{CY type iso}, we have that the Jacobian 
$$\left(\begin{array}{c}
\frac{\partial\Phi^{(0,1)}}{\partial t_1}(t) \\
\vdots\\
\frac{\partial\Phi^{(0,1)}}{\partial t_N}(t)
\end{array}\right)
$$
is non-degenerate for any $t\in \T$, which implies that the holomorphic map $\Psi:\, \T \to \C^{N}$ defines a global affine structure on $\T$.
\end{proof}

\begin{corollary}
There exists a global section of the Hodge bundle $\mathcal F^{k}$ on the Teichm\"uller space $\T$ of Calabi--Yau type manifolds,
\begin{eqnarray}
\Omega_{(n-k)}(z^{c})&=&\eta_{(n-k)}+z^{c}\cdot \eta_{(n-k+1)}+\sum_{i\ge 2}\Phi^{(n-k,n-k+i)}(z^{c})\cdot \eta_{(n-k+i)},\label{CY type section}
\end{eqnarray}
where $\eta=\{\eta_{(n-k)}^T,\eta_{(n-k+1)}^T,\cdots,\eta_{(k)}^T\}^T$ is the adapted basis of the Hodge decomposition
$$H=\H^{k,n-k}(X_{t_{o}})\oplus \H^{k-1,n-k+1}(X_{t_{o}})\oplus \cdots \oplus \H^{n-k,k}(X_{t_{o}})$$
 at the base point $t_{o}$, and $z^{c}(t)=\Phi^{(0,1)}(t)$, $t\in \T$, is the global affine coordinates given by Theorem \ref{affine structure}.
\end{corollary}

\begin{remark}
For Calabi--Yau manifolds, an analogous expansion formula was first obtained
by deformation-theoretic methods; see \cite{LRY}.  That construction gives a
local, or large-range, expansion.  The corollary above extends this formula to
Calabi--Yau type manifolds and gives a global expansion on the Teichm\"uller
space, with respect to the global affine coordinates $z^{c}=\Phi^{(0,1)}$
constructed in this paper.

From the expression \eqref{CY type section}, one sees that the constant term
and the first order term of the section of the Hodge bundle are explicit,
while the higher order terms lie in
\[
\H^{k-2,n-k+2}(X_{t_o})
\oplus
\H^{k-3,n-k+3}(X_{t_o})
\oplus
\cdots
\oplus
\H^{n-k,k}(X_{t_o}) .
\]
This normal form is useful for computing the Hodge metric and the curvature of
Hodge bundles.  In the Calabi--Yau case, such curvature computations are
closely related to Weil--Petersson geometry, Yukawa couplings, and special
geometry in mathematical physics.  The formula above suggests an analogous
framework for Calabi--Yau type manifolds.
\end{remark}

\begin{corollary}
The canonical bundle $\omega_{\T}$ of the Teichm\"uller space $\T$ of Calabi--Yau type manifolds is trivial.
\end{corollary}
\begin{proof}
Let $(z_{1},\cdots, z_{N})$ be the global coordinates on $\C^{N}$.
Then the global affine structure $\Psi:\,\T \to \C^{N}$ gives a global section 
$$\Psi^{*}(dz_{1}\cdots dz_{N})$$
of the canonical bundle $\omega_{\T}$ over $\T$, which is nowhere vanishing. Hence $\omega_{\T}$ is trivial.
\end{proof}

\begin{remark}
In \cite{RamanVafa24}, Ramana and Vafa introduce the notion of a \emph{marked moduli space}, defined as the parameter space of a physical theory together with all of its observables. 
In geometric situations, this notion should be the Teichm\"uller space defined in this paper, which has long served as the natural parameter space for marked geometric structures. 
They formulate two Swampland conjectures concerning the geometry of such spaces: first, that any marked moduli space should be contractible; and second, that with respect to the physical metric, there should exist a unique shortest path connecting any pair of points. 
Evidence for these conjectures is provided in theories with eight or more supercharges.

From this perspective, the existence of a global affine structure on the Teichm\"uller space $\T$ of Calabi--Yau type manifolds, together with the triviality of its canonical bundle $\omega_{\T}$, may be regarded as supporting evidence for these expectations. 
These geometric features suggest a rigidity phenomenon compatible with contractibility and geodesic uniqueness, and may offer useful input toward a mathematical understanding of the conjectures in the Calabi--Yau setting.
\end{remark}

\vspace{+12 pt}

%
%

\end{document}